\documentclass{article} 
\usepackage{hyperref,url}
\usepackage{natbib}
\usepackage{mathrsfs} 
\usepackage{amsmath,amsbsy,amsfonts,amssymb,amsthm,dsfont,units}
\usepackage{algorithm,algorithmic}
\usepackage{mathtools}
\usepackage{color,cases}
\usepackage{tikz}
\usetikzlibrary{calc,shapes}
\usepackage{caption}
\usepackage{subcaption}
\usepackage[utf8]{inputenc}
\usepackage[english]{babel}
\usepackage{multirow}
\usepackage{bbm}
\usepackage{xr}
\usepackage{fullpage}
\usepackage{graphicx}
\usepackage{enumitem}

\def\norm#1{\mathopen\| #1 \mathclose\|}

\newcommand{\ignore}[1]{}

\def\reals{{\mathbb R}}

\newcommand\Lcal{\mathcal{L}}
\newcommand\Acal{\mathcal{A}}

\newcommand\Rcal{\mathcal{R}}
\newcommand\Qcal{\mathcal{Q}}
\newcommand\Tcal{\mathcal{T}}
\newcommand\Hcal{\mathcal{H}}
\newcommand\Zcal{\mathcal{Z}}

\def\mP{{\mathcal P}}

\def\mG{{\mathcal G}}

\def\bold0{\mathbf{0}}

\def\bB{\mathbf{B}}

\def\bA{\mathbf{A}}
\def\bH{\mathbf{H}}
\def\bI{\mathbf{I}}
\def\bJ{\mathbf{J}}

\def\bU{\mathbf{U}}

\def\bB{{B}}
\def\bA{\mathbf{A}}
\def\bB{\mathbf{B}}
\def\bG{\mathbf{G}}
\def\bI{\mathbf{I}}
\def\bJ{\mathbf{J}}
\def\bM{\mathbf{M}}

\def\bT{\mathbf{T}}

\def\eps{\varepsilon}
\def\epsilon{\varepsilon}

\newcommand{\defeq}{\stackrel{\text{def}}{=}}

\newcommand{\braces}[1]{\left\{#1\right\}}
\newcommand{\pa}[1]{\left(#1\right)}

\newcommand{\bra}[1]{\left[#1\right]}
\newcommand{\abs}[1]{\left|#1\right|}

\DeclareMathOperator{\argmin}{argmin}

\newtheorem{theorem}{Theorem}[section]
\newtheorem{corollary}[theorem]{Corollary}

\newtheorem{lemma}[theorem]{Lemma}

\theoremstyle{definition}

\newcommand\K{\mathcal{K}}

\newcommand\Vcal{\mathcal{V}}
\newcommand\Wcal{\mathcal{W}}
\newcommand\Fcal{\mathcal{F}}

\newcommand\inner[1]{\langle #1 \rangle}

\newcommand\grad{\nabla}
\newcommand\hess{\nabla^2}
\newcommand\third{\nabla^3}
\newcommand\fourth{\nabla^4}
\newcommand\nab[1]{\nabla^#1}

\def\epsa{\tilde{\eps}_{aam}}
\def\epsr{\tilde{\eps}_{rs}}
\def\epsf{\tilde{\eps}_{fs}}
\def\epsc{\tilde{\eps}_{curr}}
\def\epsch{\hat{\eps}_{curr}}
\def\rhoin{\rho_{\text{init}}}

\def\tdelta{\tilde{\delta}}

\def\hrho{\hat{\rho}}
\def\hzeta{\hat{\zeta}}

\def\nnz{\text{nnz}}

\def\GO{\text{GO}}
\def\LSS{\text{LSS}}

\def\bcdot{\boldsymbol{\cdot}}

\newcommand{\fastq}{\mathsf{FastQuartic}}
\newcommand{\rhosearch}{\mathsf{RhoSearch}}
\newcommand{\aam}{\mathsf{ApproxAuxMin}}

\title{Fast minimization of structured convex quartics}
\author{Brian Bullins\\Princeton University}
\begin{document}

\maketitle

\begin{abstract}
We propose faster methods for unconstrained optimization of \emph{structured convex quartics}, which are convex functions of the form
\begin{equation*}
f(x) = c^\top x + x^\top \bG x + \bT[x,x,x] + \frac{1}{24}\norm{\bA x}_4^4
\end{equation*}
for  $c \in \reals^d$, $\bG \in \reals^{d \times d}$, $\bT \in \reals^{d \times d \times d}$, and $\bA \in \reals^{n \times d}$ such that $\bA^\top \bA \succ 0$. In particular, we show how to achieve an $\eps$-optimal minimizer for such functions with only $O(n^{1/5}\log^{O(1)}(\Zcal/\eps))$ calls to a gradient oracle and linear system solver, where $\Zcal$ is a problem-dependent parameter.
Our work extends recent ideas on efficient tensor methods and higher-order acceleration techniques to develop a descent method for optimizing the relevant quartic functions. As a natural consequence of our method, we achieve an overall cost of $O(n^{1/5}\log^{O(1)}(\Zcal / \eps))$ calls to a gradient oracle and (sparse) linear system solver for the problem of $\ell_4$-regression when $\bA^\top \bA \succ 0$, providing additional insight into what may be achieved for general $\ell_p$-regression. Our results show the benefit of combining efficient higher-order methods with recent acceleration techniques for improving convergence rates in fundamental convex optimization problems.
\end{abstract}

\section{Introduction}
In this paper, we are interested in the unconstrained optimization problem
\begin{equation}
\min\limits_{x \in \reals^d} f(x),
\end{equation}
where $f(x)$ is a convex function of the form
\begin{equation}\label{eq:quarticform}
f(x) = c^\top x + x^\top \bG x + \bT[x,x,x] + \frac{1}{24}\norm{\bA x}_4^4
\end{equation}
for some $c \in \reals^d$, $\bG \in \reals^{d \times d}$, $\bT \in \reals^{d \times d \times d}$, and $\bA \in \reals^{n \times d}$ such that $\bA^\top \bA \succ 0$ and $\braces{a_i}_{i\in \braces{1,\dots,n}}$ are the rows of $\bA$. We will refer to functions of this form as \emph{structured convex quartics}, as we are given an explicit decomposition of the fourth-order term, i.e.,  \[\fourth f(x) = \sum\limits_{i=1}^n a_i^{\otimes 4},\quad x \in \reals^d .\]
While fast minimization of convex quadratic functions $f(x) = c^\top x + x^\top \bG x$ has been an area of significant research efforts \citep{cohen2015uniform, clarkson2017low, agarwal2017second, agarwal2017leverage}, the structured convex quartic case has been less explored.

In this work, we present a method, called $\fastq$, whose total cost to find an $\eps$-optimal minimizer is established in the following theorem.

\begin{theorem}\label{thm:mainthm} Let $f(\cdot)$ be a convex function of the form \eqref{eq:quarticform}. Then, under appropriate initialization, $\fastq$ finds a point $x_N$ such that
\[ f(x_N) - f(x^*) \leq \eps \]
with total computational cost $O(n^{1/5}\emph{\GO}\log^{O(1)}(\Zcal/\eps) + n^{1/5}\emph{\LSS}\log^{O(1)}(\Zcal/\eps))$, where $\emph{\GO} = O(\emph{\nnz}(c) + \emph{\nnz}(\bG) + \emph{\nnz}(\bT) + \emph{\nnz}(\bA))$ is the time to calculate the gradient of $f(\cdot)$, $\emph{\LSS}$ is the time to solve a (sparse) $d\times d$ linear system, and $\Zcal$ is a problem-dependent parameter.
\end{theorem}
In the case where $n \leq O\pa{d^{5(3-\omega)}}$, $\omega \sim 2.373$ being the matrix multiplication constant, and for $n \leq O(d^5)$ when the linear system is sufficiently sparse, our method improves upon (up to logarithmic factors) the previous best rate of $O(d\GO\log(dR/\eps) + d^3\log^{O(1)}(dR/\eps))$ (where $R$ is the radius of the box containing the relevant convex set), which can be achieved by using a fast cutting plane method \citep{lee2015faster}.

We believe that, in addition to improving the complexity for a certain class of convex optimization problems, our approach illustrates the possibility of using an efficient local search-type method for some more difficult convex optimization tasks, such as $\ell_4$-regression. This is in contrast to homotopy-based approaches (such as interior-point or path-following methods) \citep{nesterov1994interior, bubeck2018homotopy}, cutting plane methods \citep{lee2015faster}, and the ellipsoid method \citep{khachiyan1980ellipsoid}.

\subsection{Related work}
In the general case, it has been shown to be NP-hard to find the global minimizer of a quartic polynomial \citep{murty1987some, parrilo2003minimizing}, or even to decide if the quartic polynomial is convex \citep{ahmadi2013np}. However, in this paper we are able to bypass these hardness results by guaranteeing the convexity of $f(\cdot)$.

In terms of optimization for higher-order smooth convex functions, for functions whose Hessian is $L_2$-Lipschitz, \cite{monteiro2013accelerated} achieve an error of $O(1/k^{7/2})$ after $\tilde{O}(k)$ calls to a second-order Taylor expansion minimization oracle. Lower bounds have been established for the oracle complexity of higher-order smooth functions, \citep{arjevani2018oracle, agarwal2018lower} which match the rate of \cite{monteiro2013accelerated} for $p=2$, and recent progress has been made toward tightening these bounds.

Some recent work from \cite{gasnikov2018global}, only available in Russian, establishes near-optimal rates for higher-order smooth optimization, though to the best of our understanding, it appears that the paper does not provide an explicit guarantee for the line search procedure. More recently, two independent works \citep{jiang2018optimal, bubeck2018near}, published on the arxiv over the past few days, establish near-optimal rates for optimization of functions with higher-order smoothness, under an oracle model, along with an analysis of the binary search procedure. In this paper, while we consider only the case for $p=3$, we go beyond the oracle model to establish an end-to-end complexity based on efficient approximations of tensor methods \citep{nesterov2018implementable}. Furthermore, while our paper also relies on a careful handling of the binary search procedure, our approach requires the more general setting of higher-order smoothness with respect to matrix-induced norms, which does not appear to follow immediately from \cite{jiang2018optimal, bubeck2018near}.

\section{Setup}
Let $\bB \in \reals^{d\times d}$ be a symmetric positive-definite matrix, i.e., $\bB \succ 0$. We let $\norm{\bM} \defeq \lambda_{\max}(\bM)$ for a matrix $\bM$, and we denote the minimizer as $x^* \defeq \argmin\limits_{x\in \reals^d} f(x)$. For any vector $v \in \reals^d$, we define its matrix-induced norm (w.r.t. $\bB$) as $\norm{v}_\bB \defeq \sqrt{v^\top \bB v}$. Throughout the paper, we will let $\hat{r}_{\bB}(x,y) \defeq \norm{x-y}_\bB$. We say a differentiable function $f(\cdot)$ is \emph{$\mu_p$-uniformly convex (of degree $p$)} with respect to $\norm{\cdot}_\bB$ if, for all $x, y \in \reals^d$,
\[ f(y) \geq f(x) + \inner{\grad f(x), y-x} + \frac{\mu_p}{p}\norm{y-x}_\bB^p. \]
Note that for $p=2$ and $\bB = \bI$, this definition captures the standard notion of strong convexity. As we shall see, since our aim is to minimize structured quartic functions, we will be concerned with this definition for $p=4$ and $\bB = \bA^\top\bA$.

A related notion is that of (higher-order) smoothness. Namely, we say a $p$-times differentiable function $f(\cdot)$ is \emph{$L_p$ smooth (of degree $p$)} w.r.t. $\norm{\cdot}_\bB$ if the $p$-th differential is $L_p$ Lipschitz continuous, i.e., for all $x, y \in \reals^d$,
\begin{equation*}
    \norm{\nab{p} f(y) - \nab{p} f(x)}_{\bB}^* \leq L_p\norm{y - x}_\bB,
\end{equation*}
where we define
\[\norm{\nab{p} f(y) - \nab{p} f(x)}_{\bB}^* \defeq \max\limits_{h : \norm{h}_\bB \leq 1} \Bigl|\nab{p} f(y) [h]^{p} - \nab{p} f(x) [h]^{p}\Bigr| \ . \]
Again, since we our concerned with quartic functions, we will later show how $f(\cdot)$ is $L_3$ smooth with respect to the appropriate norm.

For $f(\cdot)$ that are $L_3$ smooth w.r.t. $\norm{\cdot}_\bB$, we also have that, for all $x, y \in \reals^d$,
\begin{equation}\label{eq:smoothgradineq}
\norm{\grad f(y) - \grad \Phi_{x,\bB}(y)}_{\bB^{-1}} \leq \frac{L_3}{6}\norm{y-x}_\bB^3,
\end{equation}
\begin{equation}
\norm{\hess f(y) - \hess \Phi_{x,\bB}(y)}_{\bB}^* \leq \frac{L_3}{2}\norm{y-x}_\bB^2.
\end{equation}

It will eventually become necessary to handle the set of all points that might be reached by our method, starting from an initial point $x_0$. To that end, we consider the following objects, beginning with the set
\begin{equation}\label{eq:optdistset}
\K \defeq \braces{x : \norm{x - x_0}_\bB^2 \leq 4\norm{x_0 - x^*}_\bB^2}.
\end{equation}
Given this set, we now consider the maximum function value attained over $\K$, i.e.,
$\Fcal \defeq \max\limits_{x \in \K} f(x).$
Finally, we let 
\begin{equation}\label{eq:pdef}
\mP \defeq \max\limits_{x,y\in \Lcal} \norm{x-y}_\bB^2,
\end{equation}
where $\Lcal \defeq \braces{x : f(x) \leq \Fcal}$. We may also define $\mG \defeq \max\limits_{x \in \Lcal} \norm{\grad{f(x)}}_{\bB^{-1}}^2$. We note that, since $f(\cdot)$ is $L_3$ smooth, $\mP$ is a problem-dependent parameter, i.e., it depends on $c$, $\bG$, $\bT$, and $\bA$. As we will later show, the dependence on $\mP$ in the final convergence rate will only appear as part of logarithmic factors.

\subsection{Properties of convex quartic functions}
Throughout the paper, following the conventions of \cite{nesterov2018implementable}, we will let
\begin{equation}
\Phi_{x,p}(y) \defeq f(x) + \sum\limits_{i=1}^p \frac{1}{i!}\nab{p} f(x)[y-x]^i, \ p \geq 1
\end{equation}
denote the $p$-th order Taylor approximation of $f(\cdot)$, centered at $x$. Furthermore, for $f(\cdot)$ that is $L_p$ smooth, we define a model function
\begin{equation}
\Omega_{x,p,\bB}(y) \defeq \Phi_{x,p}(y) + \frac{2p L_p}{(p+1)!}\norm{y-x}_\bB^{p+1}.
\end{equation}
 As we are only concerned with functions that are $L_3$ smooth, we will drop the $p$ subscript to define $\Phi_{x}(y) \defeq \Phi_{x,3}(y)$ and 
\begin{equation}\label{eq:omegadef}
\Omega_{x,\bB}(y) \defeq \Omega_{x,3,\bB}(y) = \Phi_{x}(y) + \frac{L_3}{4}\norm{y-x}_\bB^{4}.
\end{equation}
Note that $\Omega_{x,\bB}(y)$ is $6L_3$ smooth (of degree 3) w.r.t $\norm{\cdot}_\bB$. The following theorem illustrates some useful properties of the model $\Omega_{x,\bB}(\cdot)$.
\begin{theorem}[\cite{nesterov2018implementable}, Theorem 1, for $M=2L_3$]\label{thm:modelhess}
Suppose $f(\cdot)$ is convex, 3-times differentiable, and $L_3$ smooth (of degree 3). Then, for any $x, y \in \reals^d$, we have
\begin{equation*}
0 \preceq \hess f(y) \preceq \hess \Phi_{x}(y) + \frac{L_3}{2}\norm{y-x}_\bB^2 \bB.
\end{equation*}
Moreover, for all $y \in \reals^d$,
\begin{equation}\label{eq:modelfuncvalbound}
f(y) \leq \Omega_{x,\bB}(y).
\end{equation}
\end{theorem}

With this representation of the model function $\Omega_{x,\bB}(\cdot)$ in hand, we let
\begin{equation}\label{eq:modelargmin}
T_{\bB}(x) \defeq \argmin\limits_{y \in \reals^d} \Omega_{x,\bB}(y)
\end{equation}
denote a minimizer of the fourth-order model, centered at $x$. The following lemma concerning $\Omega_{x,\bB}(\cdot)$, which will later prove useful, establishes a relaxed version of eq. (2.13) from \cite{nesterov2018implementable}.
\begin{lemma}\label{lem:gradineqapprox}
Let $\eps > 0$, and let $T_{\bB}(\cdot)$ be as in \eqref{eq:modelargmin}. Then, for all $x,y \in \reals^d$,
\begin{equation}\label{eq:approxinnerineq}
\inner{\grad f(x), y - x} \geq \frac{1}{2L_3 \hat{r}_{\bB}^{2}(x,y)} \norm{\grad f(x)}_{\bB^{-1}}^2 + \frac{3L_3}{8} \hat{r}_{\bB}^{4}(x,y) - \frac{2Z(x,y)W(x,y)\norm{x - T_\bB(y)}_\bB}{2L_3\hat{r}_{\bB}^2(x,y)},
\end{equation}
where
\begin{equation*}
Z(x,y) \defeq \norm{\grad f(x) + L_3\hat{r}_{\bB}^2(x,y)\bB(x - y)}_{\bB^{-1}},
\end{equation*}
\begin{equation*}
W(x,y) \defeq \pa{\norm{\bB^{-1/2}}^2\norm{\bH(x,y)}\norm{\bB^{-1}} + L_3\norm{x - T_\bB(y)}_\bB^2},
\end{equation*}
and
\begin{equation*}
\bH(x,y) \defeq \hess \Omega_{y,\bB}(T_\bB(y)) + \frac{1}{2}\third \Omega_{y,\bB}(T_\bB(y))[x - T_\bB(y)].
\end{equation*}
\end{lemma}
In order to get a handle on the regularity properties of $f(\cdot)$, we establish its smoothness and uniform convexity parameters w.r.t. $\norm{\cdot}_{\bA^\top \bA}$.
\begin{lemma}[$L_3$ smoothness]\label{lem:smooth}
Suppose $f(\cdot)$ is of the form \eqref{eq:quarticform}. Then, for all $x, y \in \reals^d$,
\begin{equation}
    \norm{\third f(y) - \third f(x)}_{\bA^\top \bA}^* \leq \norm{y-x}_{\bA^\top \bA}.
\end{equation}
\end{lemma}
\begin{lemma}[$\mu_4$ uniform convexity]\label{lem:unifconvex}
Suppose $f(\cdot)$ is of the form \eqref{eq:quarticform}. Then, for all $x, y \in \reals^d$,
\begin{equation}\label{eq:unifconvex}
f(y) \geq f(x) + \inner{\grad f(x), y-x} + \frac{n}{72}\norm{y-x}_{\bA^\top \bA}^4.
\end{equation}
\end{lemma}
We may also observe that $\Omega_{x,\bB}(\cdot)$ is uniformly convex w.r.t. $\norm{\cdot}_\bB$.
\begin{lemma}\label{lem:omegaunifconvex}
For all $y, z \in \reals^d$,
\begin{equation}\label{eq:omegaunifconvex}
\Omega_{x,\bB}(z) \geq \Omega_{x,\bB}(y) + \inner{\grad \Omega_{x,\bB}(y), z-y} + \frac{L_3}{12}\norm{z-y}_\bB^4.
\end{equation}
\end{lemma}

\section{Minimizing structured convex quartics}
In order to show an overall convergence rate for minimizing structured convex quartics, we shall see that the desired algorithm would be to find an exact minimizer of $\Omega_{x,\bB}(\cdot)$, for some $x$ at each iteration of the main algorithm. Thus, one of our main challenges will be to show that an approximate minimizer of $\Omega_{x,\bB}(\cdot)$ is sufficiently accurate for the rest of the algorithm. To that end, we begin by considering the auxiliary minimization problem, for which our method $\aam$ converges at a linear rate to an $\epsa$-optimal minimizer. With this approximate minimizer in hand, we find that, when taking $\epsa$ small enough, it provides a sufficiently accurate solution to be used as part of a binary search procedure, called $\rhosearch$. This approach is needed for finding an appropriate value $\rho_k$ which meets a certain approximation criterion.

Finally, once we have found an valid choice of $\rho_k$ and its corresponding $x_{k+1}$, we show how they can be used as part of our main method, called $\fastq$, to lead to a final solution $x_N$ such that $f(x_N) - f(x^*) \leq \eps$ in $O(\kappa_4^{1/5}\log(1/\eps))$ iterations of $\fastq$. Furthermore, each of these iterations requires some polylogarithmic factors incurred by $\rhosearch$ and $\aam$.

\subsection{Approximate auxiliary minimization}

To begin, we consider the auxiliary minimization problem $\min\limits_{h \in \reals^d} \Gamma_{x,\bB}(h)$, where 
\begin{equation*}
\Gamma_{x,\bB}(h) \defeq \inner{\grad f(x), h} + \frac{1}{2}h^\top \hess f(x) h + \frac{1}{6} \third f(x)[h]^3 + \frac{L_3}{4}\norm{h}_\bB^4.
\end{equation*}
Note that $\Gamma_{x,\bB}(h)$ is equivalent to $\Omega_{x,\bB}(y)$, up to a change of variables. Our aim is to establish a minimization procedure which returns an $\eps$-optimal solution in $O(\log(\Acal/\epsa))$ iterations, where $\Acal$ is a problem-dependent parameter. Furthermore, each iteration is dominated by $O(\log^{O(1)}(1/\epsa))$ calls to a (sparse) linear system solver. This subroutine, which we call $\aam$, is described in Section 5 of \cite{nesterov2018implementable} and is necessary for returning an approximate minimizer of $\Omega_{x,\bB}(\cdot)$. The approach involves showing that the auxiliary function is relatively smooth and convex \citep{lu2018relatively}, and further that each iteration of the method for minimizing such a  function reduces to a minimization problem of the form 
\begin{equation}\label{eq:lambdamin}
-\min\limits_{\lambda > 0} w(\lambda),
\end{equation}
where \begin{equation*}
w(\lambda) \defeq \frac{\lambda^2}{2} + \frac{1}{2}\inner{(\sqrt{2}\lambda\bB + \hess f(x))^{-1}c_t, c_t}
\end{equation*}
and
\begin{equation*}
c_t \defeq \grad \Gamma_{x, \bB}(h_t) = \grad f(x) + \hess f(x)h_t + \third f(x)[h_t]^2 + L_3\norm{h_t}_\bB^2\bB h_t.
\end{equation*}
As noted by \cite{nesterov2018implementable}, this minimization problem is both one-dimensional and strongly convex, and so we may achieve global linear convergence. Taken together with the relative smoothness and convexity of $\Gamma_{x,\bB}(\cdot)$, we have the following theorem.
\begin{theorem}[\cite{nesterov2018implementable}, eq.(5.9) $(\tau=\sqrt{2})$. See also: \cite{lu2018relatively}, Theorem 3.1]\label{thm:relsmoothrate}
For all $h_t$, $K \geq t \geq 0$, generated by $\aam(y_k, \epsa)$ (Algorithm \ref{alg:auxmin}), we have that
\begin{equation*}
\Gamma_{y_k,\bB}(h_t) - \Gamma_{y_k,\bB}(h^*) \leq \frac{\alpha}{\pa{\frac{\sqrt{2}+1}{2}}^t-1}\quad,
\end{equation*}
where $h^* \defeq \argmin_{h\in\reals^d}  \Gamma_{y_k,\bB}(h)$ and $\alpha \defeq \frac{1}{\sqrt{2}}(h_0-h^*)^\top\hess f(y_k)(h_0-h^*) + \frac{\sqrt{2}L_3}{4}\norm{h_0 - h^*}_\bB^4$.
\end{theorem}

\begin{algorithm}[h]
	\caption{$\aam$}

	\begin{algorithmic}\label{alg:auxmin}
		\STATE {\bfseries Input:} $y_k$, $\epsa > 0$, $K = O(\log(\Acal/\epsa))$, $h_0 = 0$.
		\FOR{$t=0$ {\bfseries to} $K$}
		\STATE $c_t \defeq \grad f(y_k) + \hess f(y_k)(h_t) + \third f(y_k)[h_t]^2 + C\norm{h_t}_\bB^2\bB (h_t)$
		\STATE$h_{t+1} = \argmin\limits_{h \in \reals^d}\braces{\inner{c_t, h - h_t} + \frac{1}{\sqrt{2}}(h-h_t)^\top\hess f(y_k)(h-h_t) + \frac{\sqrt{2}L_3}{4}\norm{h - h_t}_\bB^4}$
		\ENDFOR
		\RETURN $x_{k+1} = y_k + h_{K}$
	\end{algorithmic}
\end{algorithm}

\begin{corollary}\label{cor:auxminrate} Let $x_{k+1} = y_k+h_{K}$ be the output from $\aam(y_k, \epsa)$, for $y_k \in \Lcal$ and $K = O(\log(\Acal/\epsa))$, where $\Acal \defeq 1 + \max\limits_{z \in \Lcal} \frac{1}{\sqrt{2}}(T_\bB(z) - z)^\top\hess f(z)(T_\bB(z) - z) + \frac{\sqrt{2}L_3}{4}\norm{T_\bB(z) - z}_\bB^4$. Then
\[\Omega_{y_k, \bB}(x_{k+1}) - \Omega_{y_k, \bB}(T_{\bB}(y_k)) \leq \epsa,\]
where each iteration requires time proportional to evaluating $f(\cdot)$ in order to compute $c_t$, as well as $O(\log^{O(1)}(1 / \epsa))$ calls to a (sparse) linear system solver.
\end{corollary}
\begin{proof}
We first note that $T_{\bB}(y_k) = y_k + h^*$, and so $\Omega_{y_k,\bB}(x_{k+1}) - \Omega_{y_k,\bB}(T_{\bB}(y_k)) = \Gamma_{y_k,\bB}(h_t) - \Gamma_{y_k,\bB}(h^*)$. As observed by \cite{nesterov2018implementable} (see also: Appendix A in \cite{agarwal2017finding}), $c_t$ can be calculated in time proportional to the cost of evaluating $f(\cdot)$, which takes time $O(\nnz(c) + \nnz(\bG) + \nnz(\bT) + \nnz(\bA))$ for $f(\cdot)$ of the form \eqref{eq:quarticform}. In addition, \cite{nesterov2018implementable} notes that \eqref{eq:lambdamin} can be found by any reasonable linearly convergent procedure, and so given access to the gradient of $w(\lambda)$, this problem can be optimized (to sufficiently small error) in $O(\log^{O(1)}(1/\epsa))$ calls to a gradient oracle. Since
\begin{equation*}
\frac{d}{d\lambda}w(\lambda) = \lambda - \frac{\sqrt{2}}{2}c_t^\top(\sqrt{2}\lambda\bB + \hess f(x))^{-1}\bB(\sqrt{2}\lambda\bB + \hess f(x))^{-1}c_t,
\end{equation*}
 calculating the gradient requires $O(\LSS)$ time.

 Finally, since $K = O(\log(\Acal/\epsa))$, by our choice of $\Acal$, it follows from Theorem \ref{thm:relsmoothrate} that 
\begin{equation*}
\Omega_{y_k, \bB}(x_{k+1}) - \Omega_{y_k, \bB}(T_{\bB}(y_k)) \leq \epsa.
\end{equation*}

\end{proof}

As we shall see, it will become necessary to handle the approximation error from $\aam$, and so we provide the following lemmas to that end.
\begin{lemma}\label{lem:approxmindist}
Let $\eps > 0$, let $x_{k+1}$ be as output by $\aam(y_k, \epsa)$, and let $T_{\bB}(y_k)$ be as in \eqref{eq:modelargmin}. Then,
\begin{equation*}
\norm{x_{k+1}-T_{\bB}(y_k)}_\bB \leq \pa{\frac{12\epsa}{L_3}}^{1/4}.
\end{equation*}
\end{lemma}

\begin{proof}
By Lemma \ref{lem:omegaunifconvex}, we know that
\begin{align*}
\Omega_{y_k,\bB}(x_{k+1}) - \Omega_{y_k,\bB}(T_{\bB}(y_k)) &\geq  \inner{\grad \Omega_{y_k,\bB}(T_{\bB}(y_k)), x_{k+1}-T_{\bB}(y_k)} + \frac{L_3}{12}\norm{x_{k+1}-T_{\bB}(y_k)}_\bB^4 \\
& = \frac{L_3}{12}\norm{x_{k+1}-T_{\bB}(y_k)}_\bB^4,
\end{align*}
and so it follows from Corollary \ref{cor:auxminrate} that
\begin{equation*}
\norm{x_{k+1}-T_{\bB}(y_k)}_\bB \leq \pa{\frac{12\epsa}{L_3}}^{1/4}.
\end{equation*}
\end{proof}

\begin{lemma}\label{lem:epsagradineq}
Let $x_{k+1} = \aam(y_k, \epsa)$. Then, 
\begin{equation*}
\inner{\grad f(x_{k+1}), y_k - x_{k+1}} \geq \frac{1}{2L_3 \hat{r}_{\bB}^{2}(x_{k+1},y_k)} \norm{\grad f(x_{k+1})}_{\bB^{-1}}^2 + \frac{3L_3}{8} \hat{r}_{\bB}^{4}(x_{k+1},y_k) - \frac{3Z(x_{k+1},y_k)W(x_{k+1},y_k)\epsa^{1/4}}{L_3^{5/4}\hat{r}_{\bB}^2(x_{k+1},y_k)}\ .
\end{equation*}
\end{lemma}

\begin{proof}
The result follows from Lemmas \ref{lem:gradineqapprox} and \ref{lem:approxmindist}.
\end{proof}

\begin{lemma}\label{lem:zetadiffbound} Let $x_{k+1}$ be the output from $\aam(y_k, \epsa)$ for $y_k \in \Lcal$. In addition, let $r(y_k) \defeq \norm{T_{\bB}(y_k)-y_k}_\bB$. Then,
\begin{equation*}
\abs{\hat{r}(x_{k+1},y_k)^2 - r(y_k)^2} \leq 6\pa{\frac{\epsa}{L_3}}^{1/4}\mP^{1/2} + \pa{\frac{12\epsa}{L_3}}^{1/2}.
\end{equation*}
\end{lemma}

\begin{proof}
Let $\beta \defeq x_{k+1} - T_{\bB}(y_k)$. We have that
\begin{align*}
\abs{\hat{r}(x_{k+1},y_k)^2 - r(y_k)^2} &= \abs{\norm{T_{\bB}(y_k) - y_k}_\bB^2 - \norm{x_{k+1} - y_k}_\bB^2}\\
&= \abs{\norm{T_{\bB}(y_k) - y_k}_\bB^2 - \norm{T_{\bB}(y_k) + \beta - y_k}_\bB^2}\\
&= \abs{\norm{T_{\bB}(y_k) - y_k}_\bB^2 + 2\inner{\beta, \bB(T_{\bB}(y_k)-y_k)} + \norm{\beta}_\bB^2 - \norm{T_{\bB}(y_k) - y_k}_\bB^2}\\
&\leq 2\norm{\beta}_\bB\norm{T_{\bB}(y_k)-y_k}_\bB + \norm{\beta}_\bB^2.
\end{align*}
Now, by Lemma \ref{lem:approxmindist}, we know that $\norm{\beta}_\bB \leq \pa{\frac{12\epsa}{L_3}}^{1/4}$, and so it follows from the definition of $\mP$ that
\begin{equation*}
\abs{\hat{r}(x_{k+1},y_k)^2 - r(y_k)^2} \leq 6\pa{\frac{\epsa}{L_3}}^{1/4}\mP^{1/2} + \pa{\frac{12\epsa}{L_3}}^{1/2}.
\end{equation*}
\end{proof}

\subsection{Search procedure for finding $\rho_k$}
In this section, we establish the correctness of $\rhosearch$, our subroutine for finding an appropriate choice of $\rho_k$, given $x_k, v_k$ as inputs. One of the key algorithmic components for achieving fast higher-order acceleration, as observed by \cite{monteiro2013accelerated} and \cite{nesterov2018lectures}, is to determine $\rho_k$ such that $\rho_k \approx \zeta_k(\rho_k)$, where we define 
\begin{equation}\label{eq:zeta}
\zeta_k(\rho) \defeq \norm{T_\bB(y_k(\rho)) - y_k(\rho)}_\bB^2,
\end{equation}
\begin{equation}\label{eq:ykdef}
y_k(\rho) \defeq (1-\tau_k(\rho))x_k + \tau_k(\rho) v_k,
\end{equation}
and
\begin{equation}\label{eq:taukdef}
\tau_k(\rho) \defeq \frac{2}{1 + \sqrt{1+4L_3A_k\rho}}\ .
\end{equation}
We will also need to define an approximate version
\begin{equation}\label{eq:approxzeta}
\hzeta_k(\rho) \defeq \norm{x_{k+1}(\rho) - y_k(\rho)}_\bB^2,
\end{equation} where we let $x_{k+1}(\rho) \defeq \aam(y_k(\rho), \epsa)$. We may observe that $\zeta_k(\rho)$ is continuous in $\rho$, and furthermore that there exists some $0 \leq \rho_k^* \leq \infty$ such that $\zeta_k(\rho_k^*) = \rho_k^*$, since if $\rho = 0$, then $y_k = v_k$, and if $\rho_k \rightarrow \infty$, then $y_k = x_k$. Thus, we may reduce it to a binary search problem, under an appropriate initialization. For now, we assume that at each iteration $k \geq 0$, $\rhosearch$ is given initial bounds $\rhoin^-$ and $\rhoin^+$ such that $\rhoin^- \leq \rho_k^* \leq \rhoin^+$, thus ensuring it is a valid binary search procedure. We will later show how $\fastq$ can provide $\rhosearch$ with such guarantees.

\begin{algorithm}[h]
	\caption{$\rhosearch$}

	\begin{algorithmic}\label{alg:rhosearch}
		\STATE {\bfseries Input:} $x_k$, $v_k$, $A_k$, $\rhoin^+$, $\rhoin^-$ (s.t. $\rhoin^+ \geq \rho_k^* \geq \rhoin^-$), $\epsr > 0$, $\epsa > 0$, $M = O(\log(\Rcal / \epsr))$.
		\STATE Define $\tdelta \defeq 6\pa{\frac{\epsa}{L_3}}^{1/4}\mP^{1/2} + \pa{\frac{12\epsa}{L_3}}^{1/2}$.
		\STATE $\rho^+ \gets \rhoin^+$, $\rho^- \gets \rhoin^-$
		\FOR{$t=1$ {\bfseries to} $M$}
		\STATE $\hrho = \frac{\rho^- + \rho^+}{2}$
		\STATE $\hat{a}_{k+1} = \frac{1 + \sqrt{1+4L_3A_k\hrho}}{2L_3\hrho}$ $\quad\quad\quad\pa{\implies \hat{a}_{k+1}^2 = \frac{A_k + \hat{a}_{k+1}}{L_3 \hrho}}$
		\STATE $A_{k+1} = A_k + \hat{a}_{k+1}$
		\STATE $\tau_k = \frac{\hat{a}_{k+1}}{A_{k+1}}$
		\STATE $\hat{y}_k = (1-\tau_k)x_k + \tau_k v_k$
		\STATE $\hat{x}_{k+1} \leftarrow \aam(\hat{y}_k, \epsa)$
		\IF{$\hrho > \hzeta(\hrho) + \tdelta$}
		\STATE $\rho^+ \gets \hrho$
		\ELSIF{$\hrho < \hzeta(\hrho) - \tdelta$}
		\STATE $\rho^- \leftarrow \hrho$
		\ELSE
		\RETURN $\hrho, \hat{x}_{k+1}, \hat{a}_{k+1}$
		\ENDIF
		\ENDFOR
		\RETURN $\rho^-, \hat{x}_{k+1}, \hat{a}_{k+1}$
	\end{algorithmic}
\end{algorithm}

An important part of managing this process is to limit how quickly $\zeta_k(\rho)$ can grow, as we will need to ensure a closeness in function value once our candidate bounds $\rho^-$ and $\rho^+$ are sufficiently close. The following theorem gives us precisely what we need, namely a differential inequality w.r.t. $\abs{\zeta'_k(\rho)}$.

\begin{theorem}\label{thm:rhogradbound}
Let $\zeta_k(\rho) > 0$ be as defined in \eqref{eq:zeta}, for some $y_k(\rho) \in \Lcal$. Then we have that, for all $\rho \geq \rhoin^-$,
\begin{equation*}
\abs{\zeta_k'(\rho)} \leq \frac{\Rcal}{\zeta_k(\rho)^{1/2}},
\end{equation*}
where $\Rcal$ is as defined in \eqref{eq:zetalipschitz}.
\end{theorem}

\begin{theorem}\label{thm:rhosearch}
Given $x_k,v_k \in \Lcal$, $0 < \epsr < 1$ as inputs, and $\epsa > 0$ chosen sufficiently small, the $\rhosearch$ algorithm outputs $\rho_k$ and $x_{k+1}$ such that
\begin{equation}\label{eq:epsrcondition}
(1-\epsr)\hzeta_k(\rho_k) \leq \rho_k \leq (1+\epsr)\hzeta_k(\rho_k)
\end{equation}
where $\hzeta_k(\cdot)$ is as defined in \eqref{eq:approxzeta}.
\end{theorem}

\subsection{Analyzing the convergence of FastQuartic}\label{sec:fqanalysis}

\begin{algorithm}[h]
	\caption{$\fastq$}

	\begin{algorithmic}\label{alg:fastq}
		\STATE {\bfseries Input:} $\eps > 0$, $x_0 = 0$,  $A_0 = 0$, $\bB \succ 0$, $\frac{1}{2} > \rhoin^- > 0$, $\rhoin^+ = \mP$, $\epsa > 0$, $N$.
		\STATE Define $\psi_0(x) \defeq \frac{1}{2}\norm{x - x_0}_\bB^2$, $\epsf \defeq \min{\braces{\frac{3L_3^2\mP \rhoin^-}{32 \mG}, \frac{1}{2}}}$, $\epsr \defeq \min\braces{\frac{3L_3\rhoin^- \mP^2}{16\Tcal}, \rhoin^-, \frac{1}{2}}$, $\Tcal$ as in \eqref{eq:tcaldef}.
		\FOR{$k=0$ {\bfseries to} $N$}
        \STATE $v_k = \argmin\limits_{x \in \reals^d} \psi_k(x)$
        \STATE $a_{k+1}^- = \frac{1 + \sqrt{1+4L_3A_k\rhoin^-}}{2L_3\rhoin^-}$ $\quad\quad\quad\pa{\implies \pa{{a_{k+1}^-}}^2 = \frac{A_k + a_{k+1}^-}{L_3 \rhoin^-}}$
		\STATE $A_{k+1}^- = A_k + a_{k+1}^-$
		\STATE $\tau_k^- = \frac{a_{k+1}^-}{A_{k+1}^-}$
		\STATE $y_k^- = (1-\tau_k^-)x_k + \tau_k^- v_k$
		\STATE $x_{k+1}^- \leftarrow \aam(y_k^-, \epsa)$
        \IF{$\rhoin^- > (1+\epsf)\norm{x_{k+1}^- - y_k^-}_\bB^2$}
        \RETURN $x_{k+1}^-$
        \ELSIF{$\rhoin^- \leq (1+\epsf)\norm{x_{k+1}^- - y_k^-}_\bB^2$ and $\rhoin^- > \norm{x_{k+1}^- - y_k^-}_\bB^2 - \Qcal\epsa^{1/4}$ ($\Qcal$ as defined in \eqref{eq:qcaldef})}
        \RETURN $x_{k+1}^-$
        \ELSE
        \STATE $\rho_k, x_{k+1}, a_{k+1} \leftarrow \rhosearch(x_k, v_k, A_k, \rhoin^+, \rhoin^-, \epsr, \epsa)$
		\STATE $\psi_{k+1} = \psi_k + a_{k+1}\bra{f(x_{k+1}) + \inner{\grad f(x_{k+1}), x - x_{k+1}}}$
		\ENDIF
		\ENDFOR
		\RETURN $x_{N+1}$
	\end{algorithmic}
\end{algorithm}

Having shown the correctness of the binary search procedure in $\rhosearch$, we now describe our main algorithm, called $\fastq$, and prove its correctness. Our analysis follows that of \cite{nesterov2018lectures}, though we consider the case where $f(\cdot)$ is $L_3$ smooth (as opposed to $L_2$ smooth).

We begin by proving a useful inequality concerning the estimate sequence, as is standard for analyzing accelerated methods.
An important part of $\fastq$ is to provide $\rhosearch$ with appropriate $\rhoin^+$ and $\rhoin^-$ that are valid upper and lower bounds, respectively, on $\rho_k^*$. As we will see, setting $\rhoin^+ = \mP$ will provide a sufficiently large upper bound on $\rho_k^*$. For the lower bound, we will see that, for a small enough choice of $\rhoin^-$, if it is still the case that $\rho_k^* < \rhoin^-$, then we can show that our current iterate achieves sufficiently small error, and so we are done. The following lemmas make these observations formal.
\begin{lemma}\label{lem:fsgradineq}
Let $c > 0$, $x_{k+1} = \aam(y_k, \epsa)$, where $y_k \in \Lcal$, and suppose $\rhoin^- \leq c\hat{r}_{\bB}^2(x_{k+1},y_k)$. Then,
\begin{equation*}
\inner{\grad f(x_{k+1}), y_k - x_{k+1}} \geq \frac{1}{2L_3 \hat{r}_{\bB}^{2}(x_{k+1},y_k)} \norm{\grad f(x_{k+1})}_{\bB^{-1}}^2 + \frac{3L_3}{8} \hat{r}_{\bB}^{4}(x_{k+1},y_k) - \frac{\Wcal\epsa^{1/4}}{c\rhoin^-}\ .
\end{equation*}
where $\Wcal > 0$ is some problem-dependent parameter.
\end{lemma}
\begin{proof}
The lemma follows directly from Lemma \ref{lem:epsagradineq}, since
\begin{align*}
\inner{\grad f(x_{k+1}), y_k - x_{k+1}} &\geq \frac{1}{2L_3 \hat{r}_{\bB}^{2}(x_{k+1},y_k)} \norm{\grad f(x_{k+1})}_{\bB^{-1}}^2 + \frac{3L_3}{8} \hat{r}_{\bB}^{4}(x_{k+1},y_k) - \frac{3Z(x_{k+1},y_k)W(x_{k+1},y_k)\epsa^{1/4}}{L_3^{5/4}\hat{r}_{\bB}^2(x_{k+1},y_k)}\\
&\frac{1}{2L_3 \hat{r}_{\bB}^{2}(x_{k+1},y_k)} \norm{\grad f(x_{k+1})}_{\bB^{-1}}^2 + \frac{3L_3}{8} \hat{r}_{\bB}^{4}(x_{k+1},y_k) - \frac{\Wcal\epsa^{1/4}}{c\rhoin^-},
\end{align*}
where we let
\begin{equation}\label{eq:wdef}
\Wcal \defeq \max\limits_{x,y\in \Lcal} Z(x,y)W(x,y).
\end{equation}
\end{proof}

\begin{lemma}\label{lem:abinductapprox} For any $k \geq 0$, let $A_k$, $x_k$, $v_k$, ${y_i}_{\braces{0 \leq i \leq k-1}}$ be as generated by $k$ iterations of $\fastq$ with $\epsa > 0$ chosen sufficiently small, and suppose that for all $k$ iterations, $\rhoin^- \leq (1+\epsf)\norm{x_{k+1}^- - y_k^-}_\bB^2$ and $\rhoin^- \leq \norm{x_{k+1}^- - y_k^-}_\bB^2 - \Qcal\epsa^{1/4}$ (for $\Qcal$ as in \eqref{eq:qcaldef}).
Then, we have that
\begin{equation}\label{eq:abinductapprox}
    A_kf(x_k) + B_k \leq \psi_k^* \defeq \min_{x \in \reals^d} \psi_k(x),
\end{equation}
where $B_k \defeq \frac{3L_3}{16}\sum\limits_{i=0}^{k-1} A_{i+1} \hat{r}_{\bB}^{4} (x_{i+1},y_i)$. In addition,
\begin{equation}\label{eq:fvaldec}
f(x_k) \leq \Fcal,\quad \norm{v_k-x^*}_\bB^2 \leq \norm{x_0 - x^*}_\bB^2, \quad \text{and }\quad v_k, x_k \in \Lcal.
\end{equation}

\end{lemma}

\begin{corollary}\label{cor:akupperbound}
For any $k \geq 0$, let $A_k$, $B_k$, $x_k$ be as in the previous lemma statement. Then, we have
\begin{equation*}
f(x_k) - f(x^*) \leq \frac{1}{2A_k}\norm{x_0 - x^*}_\bB^2,
\end{equation*}
and
\begin{equation*}
B_k \leq \frac{1}{2}\norm{x_0 - x^*}_\bB^2.
\end{equation*}
\end{corollary}

\begin{lemma}\label{lem:akboundrho}
For any $k \geq 0$, we have that
\begin{equation*}
A_k \geq \frac{1}{4L_3}\pa{\sum\limits_{i=0}^{k-1} \frac{1}{\rho_i^{1/2}}}^2,
\end{equation*}
and thus $A_k \geq \frac{1}{4L_3}\frac{1}{\rho_i}$, for all $i \in \braces{0,\dots,k-1}$.
\end{lemma}

\begin{lemma}\label{lem:aklowerboundp3}
For any $k \geq 1$, we have
\begin{equation}
    A_k \geq \frac{3}{256L_3\norm{x_0-x^*}_\bB^2}\pa{\frac{k+1}{2}}^{5}.
\end{equation}
\end{lemma}

\begin{theorem}\label{thm:smoothconvrate} For any $k \geq 1$, we have
\begin{equation}
    f(x_k) - f(x^*) \leq \frac{128L_3\norm{x_0-x^*}_\bB^4}{3}\pa{\frac{2}{k+1}}^{5}.
\end{equation}
\end{theorem}
\begin{proof}
By combining Corollary \ref{cor:akupperbound} with Lemma \ref{lem:aklowerboundp3}, we observe that
\begin{align*}
f(x_k) - f(x^*) &\leq \frac{1}{2A_k}\norm{x_0 - x^*}_\bB^2 \leq \frac{128L_3\norm{x_0-x^*}_\bB^4}{3}\pa{\frac{2}{k+1}}^{5}.
\end{align*}
\end{proof}

So far, we have shown the correctness in the case where, for all $k \geq 0$, $\rhoin^- \leq (1+\epsf)\norm{x_{k+1}^- - y_k^-}_\bB^2$ and $\rhoin^- \leq \norm{x_{k+1}^- - y_k^-}_\bB^2 - \Qcal\epsa^{1/4}$. However, we need to ensure correctness of the case where, for some iteration of $\fastq$, it happens that $\rhoin^- > (1+\epsf)\norm{x_{k+1}^- - y_k^-}_\bB^2$, or $\rhoin^- \leq \norm{x_{k+1}^- - y_k^-}_\bB^2 - \Qcal\epsa^{1/4}$. We handle these cases in the following theorem.
\begin{theorem}\label{thm:outofrhobound}
Suppose there is some $1 \leq i \leq N$ such that for all iterations $1\leq j < i$, $\rhoin^- \leq (1+\epsf)\norm{x_{j+1}^- - y_j^-}_\bB^2$ and $\rhoin^- \leq \norm{x_{j+1}^- - y_j^-}_\bB^2 - \Qcal\epsa^{1/4}$, and for iteration $i$, either 
\begin{align*}
	&\text{(a) }\quad \rhoin^- > (1+\epsf)\norm{x_{i+1}^- - y_i^-}_\bB^2\text{, or}\\[0.5em]
	&\text{(b) }\quad \rhoin^- \leq (1+\epsf)\norm{x_{i+1}^- - y_i^-}_\bB^2\text{ and }\rhoin^- > \norm{x_{i+1}^- - y_i^-}_\bB^2 - \Qcal\epsa^{1/4}.
\end{align*}
Then, $\fastq$ returns $x_{i+1}$ such that
\begin{equation}
f(x_{i+1}) - f(x^*) \leq 2L_3\rhoin^-\norm{x_0-x^*}_\bB^2.
\end{equation}
\end{theorem}

\section{Main results}
Now that we have established the necessary results for proving the correctness of the output from $\aam$ and $\rhosearch$, we may combine these observations with those of Section \ref{sec:fqanalysis} to prove one of the key theorems of the paper, which establishes the total cost of optimizing smooth $f(\cdot)$.

\begin{theorem}\label{thm:smoothrate} Suppose $f(x)$ is $L_3$ smooth. Then, under appropriate initialization, $\fastq$ finds a point $x_N$ such that
\begin{equation*}
f(x_N) - f(x^*) \leq \eps
\end{equation*}
in $O\pa{\pa{\frac{L_3 \norm{x_0-x^*}_\bB^4}{\eps}}^{1/5}}$ iterations, where each iteration requires $O(\log^{O(1)}(\Zcal/\eps))$ calls to a gradient oracle and linear system solver, and where $\Zcal$ is a problem-dependent parameter.
\end{theorem}

Combining Theorem \ref{thm:smoothrate} with the appropriate notion of uniform convexity, we may establish a rate of linear convergence, based on the condition number $\kappa_4 \defeq \frac{L_3}{\mu_4}$. In addition, the proof of the main theorem follows almost immediately.
\begin{theorem}\label{thm:kappa} Suppose $f(x)$ is $L_3$-smooth and $\mu_4$-uniformly convex w.r.t. $\norm{\cdot}_\bB$. Then, under appropriate initialization, $\fastq$ finds a point $x_N$ such that
\begin{equation*}
f(x_N) - f(x^*) \leq \eps
\end{equation*}
in $\tilde{O}\pa{\kappa_4^{1/5}\log(1/\eps)}$ iterations, where each iteration requires $O(\log^{O(1)}(\Zcal/\eps))$ calls to a gradient oracle and linear system solver, and where $\Zcal$ is a problem-dependent parameter.
\end{theorem}
\begin{proof}
Begin by running the $\fastq$ algorithm for $k = \left\lceil \pa{\frac{512L_3}{3\mu_4}}^{1/5} \right\rceil$ iterations. By combining Theorem \ref{thm:smoothrate} with the fact that $f(\cdot)$ is uniformly convex, we have that
\begin{equation*}
f(x_k) - f(x^*) \leq \frac{128L_3\norm{x_0-x^*}_\bB^4}{3}\pa{\frac{2}{k+1}}^{5}\leq \frac{512L_3 (f(x_0) - f(x^*))}{3\mu_4 k^5}.
\end{equation*}
It follows from our choice of $k$ that 
\begin{equation*}
f(x_k) - f(x^*) \leq \frac{f(x_0) - f(x^*)}{2}.
\end{equation*}
Because we reduce the optimality gap by a constant factor every $k$ iterations, it suffices to run FastQuartic for $N= O(\kappa_4 \log(1/\eps))$ iterations to achieve a point $x_N$ such that
\begin{equation*}
f(x_N) - f(x^*) \leq \eps.\qedhere
\end{equation*}
\end{proof}
\noindent Having developed all of the necessary results, we may now prove our main theorem.
\begin{proof}[Proof of Theorem \ref{thm:mainthm}]
The proof follows by combining Theorem \ref{thm:kappa} with Lemmas \ref{lem:smooth} and \ref{lem:unifconvex}.
\end{proof}
As a consequence of our result, we have the following guarantee for the problem of $\ell_4$-regression, which improves upon (up to logarithmic factors) the $O^*(n^{1/4})$ calls to a sparse linear system solver as shown by \cite{bubeck2018homotopy}, when $\bA^\top \bA \succ 0$ and $\bA$ is sparse.
\begin{corollary}\label{cor:maincor} For the problem of $\ell_4$-regression, i.e., problems of the form
\begin{equation*}
\min_{x \in \reals^d} f(x) = c^\top x + \norm{\bA x - b}_4^4,
\end{equation*}
for $c\in \reals^d$, $b \in \reals^n$, $\bA \in \reals^{n \times d}$ such that $\bA^\top \bA \succ 0$,  the $\fastq$ algorithm finds, under appropriate initialization, a point $x_N$ such that
\begin{equation*} f(x_N) - f(x^*) \leq \eps
\end{equation*}
with $O(n^{1/5}\log^{O(1)}(\Zcal/\eps))$ calls to a gradient oracle and (sparse) linear system solver.
\end{corollary}
\begin{proof}
Note that for all $x\in\reals^d$, $\fourth f(x) = 24\sum\limits_{i=1}^n a_i^{\otimes 4}$, where $\bA = \bra{a_1 a_2 \dots a_n}^\top$. Since $f(x)$ is a quartic function, we may equivalently express it as its fourth-order Taylor expansion
\begin{align*}f(x) &= f(0) + \grad f(0)^\top x + \frac{1}{2} x^\top \hess f(0) x + \frac{1}{6} \third f(0)[x,x,x] + \frac{1}{24}\fourth f(0)[x]^{\otimes 4}\\
&= f(0) + \grad f(0)^\top x + \frac{1}{2} x^\top \hess f(0) x + \frac{1}{6} \third f(0)[x,x,x] + \norm{\bA x}_4^4,
\end{align*}
and so since $f(\cdot)$ is of the form \eqref{eq:quarticform}, for $\bA^\top \bA \succ 0$, the result follows from Theorem \ref{thm:mainthm}, and the observation that each iteration of $\aam$ requires solving a sparse linear system, if $\bA$ is sparse.
\end{proof}

\section{Conclusion}
We have presented the method $\fastq$ for efficiently minimizing structured convex quartics. Moving forward, we believe one future direction would be to explore how $\fastq$ might be a useful tool for achieving faster convergence in various other convex optimization problems. An interesting open problem would be to reduce the dependence on $n$ to $d$. We would further like to note the connection between the $\norm{\bA x}_4^4$ term in \eqref{eq:quarticform} and polynomial norms as studied by \cite{ahmadi2017polynomial}, as this perspective may prove useful as part of future work.

\section*{Acknowledgements}
The authors would like to thank Naman Agarwal, Cyril Zhang, and Yi Zhang for helpful discussions. We would especially like to thank Karan Singh for numerous enlightening discussions, as well as for help with proofreading the manuscript.

\bibliography{main}
\bibliographystyle{plainnat}

\appendix
\section{Proofs}

\subsection{Proof of Lemma \ref{lem:gradineqapprox}}
\begin{proof}
Let $x, y\in \reals^d$, let $\hat{r}_{\bB}(x,y) \defeq \norm{x-y}_\bB$, and let $\delta(x,y) \defeq \grad \Omega_{y,\bB}(x)$. Using the $L_3$ smoothness of $f(x)$, we have by \eqref{eq:smoothgradineq} and the triangle inequality that
\begin{align*}
\norm{\grad f(x) + L_3\hat{r}_{\bB}^2(x,y)\bB(x - y)} - \norm{\delta(x,y)} &\leq \norm{\grad f(x) + L_3\hat{r}_{\bB}^2(x,y)\bB(x - y) - \delta(x,y)}_{\bB^{-1}}\\
&= \norm{\grad f(x) - \grad \Phi_{y}(x)}_{\bB^{-1}}\\
&\leq \frac{L_3}{6}\hat{r}_{\bB}^3(x,y),
\end{align*}
where the last inequality follows from \eqref{eq:smoothgradineq}. Squaring both sides gives us
\begin{equation*}
\norm{\grad f(x) + L_3\hat{r}_{\bB}^2(x,y)\bB(x - y)}_{\bB^{-1}}^2 - \Delta(x,y) \leq \frac{L_3^2}{36}\hat{r}_{\bB}^6(x,y),
\end{equation*}
where
\begin{equation*}\Delta(x,y) \defeq 2Z(x,y)\norm{\delta(x,y)}_{\bB^{-1}}-\norm{\delta(x,y)}_{\bB^{-1}}^2
\end{equation*}
and
\begin{equation*}
Z(x,y) \defeq \norm{\grad f(x) + L_3\hat{r}_{\bB}^2(x,y)\bB(x - y)}_{\bB^{-1}}.
\end{equation*}
After expanding and rearranging the terms in the inequality, we arrive at
\begin{equation*}
\norm{\grad f(x)}_{\bB^{-1}}^2 + \frac{35}{36}L_3^2\hat{r}_{\bB}^6(x,y) - \Delta(x,y) \leq 2L_3\hat{r}_{\bB}^2(x,y)\inner{\grad f(x), y - x}.
\end{equation*}
Diving both sides by $2L_3\hat{r}_{\bB}^2(x,y)$ gives us
\begin{equation}\label{eq:approxdeltaineq}
\frac{\norm{\grad f(x)}_{\bB^{-1}}^2}{2L_3\hat{r}_{\bB}^2(x,y)} + \frac{35}{72}L_3\hat{r}_{\bB}(x,y)^4 - \frac{\Delta(x,y)}{2L_3\hat{r}_{\bB}^2(x,y)} \leq \inner{\grad f(x), y - x}.
\end{equation}

All that remains is to bound $\Delta(x,y)$. Note that, by \eqref{eq:smoothgradineq} and using the fact that $\grad \Omega_{y,\bB}(T_\bB(y)) = 0$,
\begin{align*}
&\norm{\grad \Omega_{y,\bB}(x) - \grad \Omega_{y,\bB}(T_\bB(y)) - \hess \Omega_{y,\bB}(T_\bB(y))[x - T_\bB(y)] - \frac{1}{2}\third \Omega_{y,\bB}(T_\bB(y))[x - T_\bB(y)]^2}_{\bB^{-1}}\\
&= \norm{\grad \Omega_{y,\bB}(x) - \hess \Omega_{y,\bB}(T_\bB(y))[x - T_\bB(y)] - \frac{1}{2}\third \Omega_{y,\bB}(T_\bB(y))[x - T_\bB(y)]^2}_{\bB^{-1}}\\
&\leq L_3\norm{x - T_\bB(y)}_\bB^3.
\end{align*}
By triangle inequality and rearranging, we have
\begin{align}
\norm{\grad \Omega_{y,\bB}(x)}_{\bB^{-1}} \leq \norm{\bH(x,y)(x - T_\bB(y))}_{\bB^{-1}} + L_3\norm{x - T_\bB(y)}_\bB^3
\end{align}
where $\bH(x,y) \defeq \hess \Omega_{y,\bB}(T_\bB(y)) + \frac{1}{2}\third \Omega_{y,\bB}(T_\bB(y))[x - T_\bB(y)]$. Note that, by our choice of $\bB$, we may write its eigendecomposition as $\bB = \bU\Lambda\bU^\top$, and we may define $\bB^{1/2} \defeq \bU\Lambda^{1/2}\bU^\top$ and $\bB^{-1/2} \defeq \bU\Lambda^{-1/2}\bU^\top$. Thus, we can then rewrite
\begin{align*}
\norm{\bH(x,y)(x - T_\bB(y))}_{\bB^{-1}} &= \norm{\bB^{-1/2}\bH(x,y)(x - T_\bB(y))}\\
&\leq \norm{\bB^{-1/2}}\norm{\bH(x,y)}\norm{x - T_\bB(y)}\\
&= \norm{\bB^{-1/2}}\norm{\bH(x,y)}\norm{\bB^{-1}}\norm{\bB^{-1/2}\bB^{1/2}(x - T_\bB(y))}\\
&\leq \norm{\bB^{-1/2}}\norm{\bH(x,y)}\norm{\bB^{-1}}\norm{\bB^{-1/2}}\norm{\bB^{1/2}(x - T_\bB(y))}\\
&= \norm{\bB^{-1/2}}^2\norm{\bH(x,y)}\norm{\bB^{-1}}\norm{x - T_\bB(y)}_\bB,
\end{align*}
and so it follows that
\begin{equation*}
\norm{\grad \Omega_{y,\bB}(x)}_{\bB^{-1}} \leq \pa{\norm{\bB^{-1/2}}^2\norm{\bH(x,y)}\norm{\bB^{-1}} + L_3\norm{x - T_\bB(y)}_\bB^2}\norm{x - T_\bB(y)}_\bB = W(x,y)\norm{x - T_\bB(y)}_\bB.
\end{equation*}
Taken together with \eqref{eq:approxdeltaineq}, we have that
\begin{align*}
\inner{\grad f(x), y - x} \geq \frac{\norm{\grad f(x)}_{\bB^{-1}}^2}{2L_3\hat{r}_{\bB}^2(y)} + \frac{35}{72}L_3\hat{r}_{\bB}(y)^4 - \frac{2Z(x,y)W(x,y)\norm{x - T_\bB(y)}_\bB}{2L_3\hat{r}_{\bB}^2(y)}.
\end{align*}

\end{proof}

\subsection{Proof of Lemma \ref{lem:smooth}.}
\begin{proof}
    Note that for all $\xi \in \reals^d$,
    \begin{align}
        \norm{\fourth f(\xi)}_\bB^* &= \max\limits_{h : \norm{h}_\bB \leq 1} \Bigl|\fourth f(\xi)[h]^4 \Bigr| = \max\limits_{h : \norm{h}_\bB \leq 1} \norm{\bA h}_4^4 \leq \max\limits_{h : \norm{h}_\bB \leq 1} \norm{\bA h}_2^4.
    \end{align}
    Setting $\bB = \bA^\top \bA$ gives us
    \[\max\limits_{h : \norm{h}_{\bA^\top \bA} \leq 1} \norm{\bA h}_2^4 \leq 1.\]
    By the mean value theorem, we have, for some $\xi$ along the line between $x$ and $y$,
    \[\frac{\norm{\third f(y) - \third f(x)}_{\bA^\top \bA}^*}{\norm{y-x}_{\bA^\top \bA}} = \norm{\fourth f(\xi)}_{\bA^\top \bA}^* \leq 1,\]
    and so it follows that
    \[\norm{\third f(y) - \third f(x)}_{\bA^\top \bA}^* \leq \norm{y-x}_{\bA^\top \bA}.\]
\end{proof}

\subsection{Proof of Lemma \ref{lem:unifconvex}.}
\begin{proof}
Following the same idea as in the proof of Lemma \ref{lem:smooth}, we note that, for all $x, y \in \reals^d$,
\[f(y) = \Phi_{x,4}(y).\]
Since $f(y)$ is convex by definition, it follows that
\[0 \preceq \hess f(y) = \hess f(x) + \third f(x)[y-x] + \frac{1}{2}\fourth f(x)[y-x, y-x].\]
Let $h = y-x$. Then, following the approach of \cite{nesterov2018implementable}, we have
\[ -\third f(x)[h] \preceq \hess f(x) + \frac{1}{2}\fourth f(x)[h, h]. \]
Since this holds for any $x, y$ (and therefore, for any direction $h$), we can replace $h$ with $\tau h$ for any $\tau > 0$ and arrive at
\[ -\tau\third f(x)[h] \preceq \hess f(x) + \tau^2\frac{1}{2}\fourth f(x)[h, h]. \]
Furthermore, we can replace $h$ by $-h$ to get
\[ \tau\third f(x)[h] \preceq \hess f(x) + \tau^2\frac{1}{2}\fourth f(x)[h, h], \]
and so after dividing through by $\tau$, we obtain
\[-\frac{1}{\tau}\hess f(x) - \frac{\tau}{2}\fourth f(x)[h, h] \preceq \third f(x)[h] \preceq \frac{1}{\tau}\hess f(x) + \frac{\tau}{2}\fourth f(x)[h, h]. \]

We may now observe that
\begin{align*}f(y) &= f(x) + \inner{\grad f(x), y-x} + \frac{1}{2}\hess f(x)[y-x, y-x] + \frac{1}{6}\third f(x)[y-x]^3 + \frac{1}{24}\fourth f(x)[y-x]^4 \\
& \geq f(x) + \inner{\grad f(x), y-x} + \pa{\frac{1}{2} - \frac{1}{6\tau}}\hess f(x)[y-x, y-x] + \pa{\frac{1}{24} - \frac{\tau}{12}}\fourth f(x)[y-x]^4.
\end{align*}
Setting $\tau = \frac{1}{3}$ gives us
\begin{align*}
f(y) &\geq f(x) + \inner{\grad f(x), y-x} + \frac{1}{72}\fourth f(x)[y-x]^4\\
&= f(x) + \inner{\grad f(x), y-x} + \frac{1}{72}\norm{\bA(y-x)}_4^4\\
&\geq f(x) + \inner{\grad f(x), y-x} + \frac{n}{72}\norm{\bA(y-x)}_2^4\\
&= f(x) + \inner{\grad f(x), y-x} + \frac{n}{72}\norm{y-x}_{\bA^\top \bA}^4,
\end{align*}
which gives us \eqref{eq:unifconvex}.
\end{proof}

\subsection{Proof of Lemma \ref{lem:omegaunifconvex}.}
\begin{proof}
We note that, for all $y, z \in \reals^d$, since $\Omega_{x,\bB}(z)$ is convex, it follows from the proof of Lemma \ref{lem:unifconvex} that
\begin{align*}
\Omega_{x,\bB}(z) &= \Omega_{x,\bB}(y) + \inner{\grad \Omega_{x,\bB}(y), z-y} + \frac{1}{2}(z-y)^\top\hess \Omega_{x,\bB}(y)(z-y) + \frac{1}{6}\third \Omega_{x,\bB}(y)[z-y]^3 + \frac{1}{24}\fourth \Omega_{x,\bB}(y)[z-y]^4\\
&\geq \Omega_{x,\bB}(y) + \inner{\grad \Omega_{x,\bB}(y), z-y} + \frac{1}{72}\fourth \Omega_{x,\bB}(y)[z-y]^4\\
&= \Omega_{x,\bB}(y) + \inner{\grad \Omega_{x,\bB}(y), z-y} + \frac{L_3}{12}\norm{z-y}_\bB^4.
\end{align*}
\end{proof}

\subsection{Proof of Theorem \ref{thm:rhogradbound}.}
\begin{proof}
Note that $\zeta_k(\rho) = (m \circ y_k)(\rho)$, where $m(y_k) = \norm{T_\bB(y_k) - y_k}_\bB^2$ and $y_k(\rho)$ is as defined in \eqref{eq:ykdef}.
Therefore, by the chain rule, we have
\begin{align*}
 \abs{\zeta_k'(\rho)} &= \abs{\bJ_\rho y_k(\rho) \grad_{y_k} m(y_k(\rho))} \\
 &\leq \norm{\bJ_\rho y_k(\rho)}_\bB\norm{\grad_{y_k} m(y_k(\rho))}_{\bB^{-1}}\\
 &\leq \lambda_{\max}(\bB^{-1})^{1/2}\norm{\bJ_\rho y_k(\rho)}_\bB\norm{\grad_{y_k} m(y_k(\rho))},
\end{align*}
where we let $\bJ$ denote the Jacobian. For $\norm{\bJ_\rho y_k(\rho)}_\bB$, we know by \eqref{eq:ykdef} and \eqref{eq:taukdef} that
\begin{equation*}
y_k(\rho) = (1-\tau_k(\rho))x_k + \tau_k(\rho) v_k
\end{equation*}
and
\begin{equation*}
\tau_k(\rho) = \frac{2}{1 + \sqrt{1+4L_3A_k\rho}}.
\end{equation*}
Thus, it follows that 
\begin{equation*}
\bJ_\rho y_k(\rho) = -\frac{d}{d\rho}\tau_k(\rho) \cdot x_k + \frac{d}{d\rho}\tau_k(\rho) \cdot v_k. 
\end{equation*}
Note that
\begin{align*}
\abs{\frac{d}{d\rho}\tau_k(\rho)} &= \frac{4L_3A_k}{(1 + \sqrt{1+4L_3A_k\rho})^2\sqrt{1+4L_3A_k\rho}} \leq \frac{4L_3A_k}{\pa{1+4L_3A_k\rho}^{3/2}}\leq \frac{1}{\rho}\quad.
\end{align*}
Taken together, this gives us that
\begin{align*}
\norm{\bJ_\rho y_k(\rho)}_\bB &\leq \abs{\frac{d}{d\rho}\tau_k(\rho)}(\norm{x_k}_\bB + \norm{v_k}_\bB) \leq \frac{\norm{x_k}_\bB + \norm{v_k}_\bB}{\rho}.
\end{align*}

To provide a bound for $\norm{\grad_{y_k} m(y_k(\rho))}$, we begin by letting $g(x, z) \defeq \Omega_{x, \bB}(z)$. We may see that $T_\bB(y_k) = \argmin\limits_{z \in\reals^d} g(y_k,z)$. As long as $\bra{\partial_z^2 g(y_k, T_\bB(y_k))}^{-1}\succ 0$, which we will see holds when $\norm{T_\bB(y_k)-y_k}_\bB > 0$, we have that, by the implicit function theorem,
\begin{equation*}
\bJ_x T_\bB(x) =  -\bra{\partial_z^2 g(x, T_\bB(x))}^{-1}\partial_x\partial_z g(x, T_\bB(x)).
\end{equation*}
Note that, since $g(x,z) = \Phi_{x}(z) + \frac{L_3}{4}\norm{z-x}_\bB^{4}$, we have
\begin{equation*}
\partial_z g(x, z) = \grad f(x) + \hess f(x)[z-x] + \frac{1}{2}\third f(x)[z-x]^2 + L_3\norm{z-x}_\bB^2 \bB(z-x),
\end{equation*}
and so it follows that
\begin{align*}
\partial_z^2 g(x, z) &= \hess f(x) + \third f(x)[z - x] + 2L_3\bB(z-x)(z-x)^\top \bB + L_3\norm{z-x}_\bB^2\bB \\
&\succeq \hess f(x) + \third f(x)[z - x] + L_3\norm{z-x}_\bB^2\bB,
\end{align*}
and 
\begin{align*}
\partial_x\partial_z g(x, z) &= \hess f(x) + \third f(x)[z-x] - \hess f(x) + \frac{1}{2}\fourth f(x)[z-x]^2 \\
&\quad\quad -\third f(x)[z-x] + 2L_3\bB(z-x)(z-x)^\top \bB - L_3\norm{z-x}_\bB^2\bB\\
&= \fourth f(x)[z-x]^2 + 2L_3\bB(z-x)(z-x)^\top \bB - L_3\norm{z-x}_\bB^2\bB.
\end{align*}
Thus,
\begin{equation}\label{eq:xzupper}
\norm{\partial_x\partial_z g(x, z)} \leq H(x,z),
\end{equation}
where \begin{equation*}
H(x,z) \defeq \norm{\fourth f(x)[z-x]^2} + 2L_3\norm{\bB(z-x)(z-x)^\top \bB} + L_3\norm{z-x}_\bB^2\norm{\bB}.
\end{equation*}
By Theorem~\ref{thm:modelhess} we have that $\hess f(x) + \third f(x)[z-x] + \frac{L_3}{2}\norm{z-x}_\bB^2\bB \succeq 0$,
and so
\begin{align*}
\partial_z^2 g(x, z) &\succeq \hess f(x) + \third f(x)[z - x] + L_3\norm{z-x}_\bB^2\bB\\
& \succeq \frac{L_3\norm{z-x}_\bB^2}{2}\bB.
\end{align*}
Thus,
\begin{equation}\label{eq:zzupper}
\norm{\bra{\partial_z^2 g(x, z)}^{-1}} \leq \frac{1}{\lambda_{\min}\pa{\bra{\partial_z^2 g(x, z))}}} \leq \frac{2}{L_3\lambda_{\min}(\bB)\norm{z-x}_\bB^2}.
\end{equation}
We may now observe that, for $m(y)$,
\begin{equation*}
\grad_{y_k} m(y_k) = 2(\bJ_{y_k} T(y_k) - \bI)\bB(T(y_k) - y_k),
\end{equation*}
and so, by standard matrix norm inequalities,
\begin{align*}
\norm{\grad_{y_k} m({y_k})} &= 2\norm{(\bJ_{y_k} T_\bB({y_k}) - \bI)\bB^{1/2}\bB^{1/2}(T({y_k}) - {y_k})} \\
&\leq 2\norm{\bJ_{y_k} T_\bB({y_k})}\bcdot\norm{\bB^{1/2}}\bcdot\norm{T({y_k}) - {y_k}}_\bB + \norm{\bB^{1/2}}\bcdot\norm{T({y_k}) - {y_k}}_\bB\\
&\leq 2\lambda_{\max}(\bB^{1/2})\pa{\norm{\bra{\partial_z^2 g({y_k}, T_\bB({y_k}))}^{-1}\partial_x\partial_z g({y_k}, T_\bB({y_k}))}\bcdot\norm{T_\bB({y_k})-{y_k}}_\bB + \norm{T({y_k}) - {y_k}}_\bB}\\
&\leq 2\lambda_{\max}(\bB^{1/2})\pa{\norm{\bra{\partial_z^2 g({y_k}, T_\bB({y_k}))}^{-1}}\bcdot\norm{\partial_x\partial_z g({y_k}, T_\bB({y_k}))}\bcdot\norm{T_\bB({y_k})-{y_k}}_\bB + \norm{T({y_k}) - {y_k}}_\bB}\\
&\leq 2\lambda_{\max}(\bB^{1/2})\pa{\frac{2H({y_k},T_\bB({y_k})) + L_3\lambda_{\min}(\bB)\norm{T({y_k}) - {y_k}}_\bB^2}{L_3\lambda_{\min}(\bB) \norm{T_\bB({y_k})-{y_k}}_\bB}}
\end{align*}
where the last inequality follows from \eqref{eq:xzupper} and \eqref{eq:zzupper}, and since $\norm{T_\bB(y_k)-y_k}_\bB > 0$ (as if $T_\bB(y_k)=y_k$, then $y_k$ is a minimizer of $f(\cdot)$).

All together, this gives us that
\begin{align*}
\abs{\zeta'(\rho)} &\leq \lambda_{\max}(\bB^{-1})^{1/2}\norm{\bJ_\rho y_k(\rho)}\norm{\grad_{y_k} m(y_k(\rho))}\\
&\leq \lambda_{\max}(\bB^{-1})^{1/2} \pa{\frac{\norm{x_k}_\bB + \norm{v_k}_\bB}{\rho}}\pa{2\lambda_{\max}(\bB^{1/2})\pa{\frac{2H(y_k(\rho),T_\bB(y_k(\rho))) + L_3\lambda_{\min}(\bB)\norm{T(y_k(\rho)) - y_k(\rho)}_\bB^2}{L_3\lambda_{\min}(\bB) \norm{T_\bB(y_k(\rho))-y_k(\rho)}_\bB}}}.
\end{align*}
Let $\Hcal \defeq \max\limits_{x,z \in \Lcal} H(x,z)$, $\rhoin^-$ be our initial lower bound on $\rho_k^*$, and $\mP$ be as in \eqref{eq:pdef}. Since $y_k(\rho) \in \Lcal$ and $\zeta(\rho) = \norm{T_\bB(y_k(\rho))-y_k(\rho)}_\bB^2$ by definition, it follows that
\begin{align*}
\abs{\zeta'(\rho)} &\leq \frac{\Rcal}{\zeta(\rho)^{1/2}},
\end{align*}
where 
\begin{equation}\label{eq:zetalipschitz}
\Rcal \defeq \frac{4\mP^{1/2}\lambda_{\max}(\bB^{1/2})\pa{2\Hcal + L_3\lambda_{\min}(\bB)\mP}}{L_3\lambda_{\min}(\bB)\rhoin^-}\ .
\end{equation}
\end{proof}

\subsection{Proof of Theorem \ref{thm:rhosearch}.}
\begin{proof}
By sufficiently small, we mean that $\epsa$ is chosen such that $\epsa \leq \min\braces{\pa{\frac{\epsr^2}{100\Qcal}}^4, \pa{\frac{\epsr^2}{100\Wcal}}^4}$, for $\Wcal$ as defined in \eqref{eq:wdef}, and for
\begin{equation}\label{eq:qcaldef}
\Qcal \defeq \pa{\frac{6\mP^{1/2}}{L_3^{1/4}} + \frac{5}{L_3^{1/2}}}.
\end{equation}
We proceed by proving the correctness of the binary search procedure. Consider $\hrho$ from the algorithm, and let $\hat{x}_{k+1}$ be the output from the call to $\aam(\hat{y}_k, \epsa)$ in the $\rhosearch$ algorithm. Then, at each iteration, one of the following three conditions must hold:
\begin{align*}
    &\text{(a) } \hrho > \hzeta_k(\hrho) + \tdelta\text{; or}\\[0.5em]
    &\text{(b) } \hrho < \hzeta_k(\hrho) - \tdelta\text{; or}\\[0.5em]
    &\text{(c) } \hzeta_k(\hrho) - \tdelta \leq \hrho \leq \hzeta_k(\hrho) + \tdelta,
\end{align*}
where 
\begin{equation*}
\tdelta \defeq 6\pa{\frac{\epsa}{L_3}}^{1/4}\mP^{1/2} + \pa{\frac{12\epsa}{L_3}}^{1/2}.
\end{equation*}
Note that, based on our choice of $\epsa$, we ensure that $\tdelta \leq \frac{\epsr^2}{4}$. Suppose condition (a) holds. Then, by Lemma \ref{lem:zetadiffbound} (with $y_k = y_k(\hrho)$), we have that $\zeta_k(\hrho) - \tdelta \leq \hzeta(\hrho)$, and so it follows that $\hrho > \zeta_k(\hrho)$. Thus, $\hrho$ is an upper bound on $\rho_k^*$, and so this proves the correctness $\rho^+$ remaining an upper bound on $\rho_k^*$ after updating $\rho^+ \leftarrow \hrho$. By similar reasoning, we may conclude that if condition (b) holds, $\hrho$ is a lower bound on $\rho_k^*$, and so $\rho^-$ remains a lower bound on $\rho_k^*$ after updating $\rho^- \leftarrow \hrho$. 

If condition (c) holds, then it must be the case that $\hzeta_k(\hrho) \geq \frac{\epsr}{2}$, since if we suppose that $\hzeta_k(\hrho) < \frac{\epsr}{2}$, this implies that $\hrho \leq \hzeta_k(\hrho) + \tdelta \leq \frac{3\epsr}{4}$. However, this is a contradicition since we ensure that $\hrho \geq \rhoin^- \geq \epsr$. Therefore, since $\tdelta \leq \frac{\epsr^2}{4} \leq \epsr\hzeta_k(\hrho)$, it follows that
\begin{equation*}
 (1-\epsr)\hzeta_k(\hrho) \leq \hrho \leq (1+\epsr)\hzeta_k(\hrho),
\end{equation*}
which means that condition \eqref{eq:epsrcondition} is met.

Based on our choice of update, anytime condition (a) or (b) holds and the update takes place, we guarantee a decrease in $\abs{\rho^+ - \rho^-}$, and so after $O(\log(\Rcal / \epsr))$ iterations, we are assured that $\abs{\rho^+ - \rho^-} \leq \frac{\epsr^{3}}{100\Rcal}$. At this point, we make use of Theorem \ref{thm:rhogradbound} to argue that $\rho^-$  must fall in the desired range, i.e., $(1-\epsr)\hzeta_k(\rho^-) \leq \rho^- \leq (1+\epsr)\hzeta_k(\rho^-)$. To show this, we first note that $\abs{\rho_k^* - \rho^-} \leq \frac{\epsr^{3}}{100\Rcal}$.
Thus, using the fact that $\zeta_k(\rho) \geq 0$, Theorem \ref{thm:rhogradbound} implies that
\begin{align*}
&\abs{\zeta_k'(\rho)(\zeta_k(\rho))^{1/2}} \leq \Rcal \implies -\Rcal \leq \zeta_k'(\rho)(\zeta_k(\rho))^{1/2} \leq \Rcal.
\end{align*}
Note that $\rho^- \leq \rho_k^*$. By integrating with respect to $\rho$, we have
\begin{equation*}
\int_{\rho_k^*}^{\rho^-} \Rcal d\rho \leq \int_{\rho_k^*}^{\rho^-}\zeta_k'(\rho)(\zeta_k(\rho))^{1/2}d\rho \leq \int_{\rho_k^*}^{\rho^-}-\Rcal d\rho.
\end{equation*}
It follows that
\begin{equation*}
\frac{2}{3}\zeta_k(\rho_k^*)^{3/2}+L(\rho^- - \rho_k^*) \leq \frac{2}{3}\zeta_k(\rho^-)^{3/2} \leq \frac{2}{3}\zeta_k(\rho_k^*)^{3/2} - \Rcal(\rho^- -\rho_k^*), 
\end{equation*}
and so we have
\begin{equation*}
\pa{\zeta_k(\rho_k^*)^{3/2}+\frac{3\Rcal}{2}(\rho^- -\rho_k^*)}^{2/3} \leq \zeta_k(\rho^-) \leq \pa{\zeta_k(\rho_k^*)^{3/2} - \frac{3\Rcal}{2}(\rho^- -\rho_k^*)}^{2/3}. 
\end{equation*}
We may now observe that
\begin{align*}
\zeta_k(\rho^-) &\leq \pa{\zeta_k(\rho_k^*)^{3/2} - \frac{3\Rcal}{2}(\rho^- -\rho_k^*)}^{2/3}\\
&= \pa{\zeta_k(\rho_k^*)^{3/2} + \frac{3\Rcal}{2}(\rho_k^* -\rho^-)}^{2/3}\\
 &\leq \pa{\zeta_k(\rho_k^*)^{3/2} + \frac{\epsr^{3}}{50}}^{2/3}\\
 &\leq \zeta_k(\rho_k^*) + \pa{\frac{1}{50}}^{2/3}\epsr^2,
\end{align*}
and so 
\begin{equation}\label{eq:rhostardiff}
\zeta_k(\rho^-) - \zeta_k(\rho_k^*) \leq \frac{\epsr^2}{10}.
\end{equation}

We again use Lemma \ref{lem:zetadiffbound} to see that
\begin{equation}\label{eq:c2diffbound}
\abs{\zeta(\rho^-) - \hzeta_k(\rho^-)} \leq 6\pa{\frac{\epsa}{L_3}}^{1/4}\mP^{1/2} + \pa{\frac{12\epsa}{L_3}}^{1/2} \leq \Qcal\epsa^{1/4},
\end{equation}
where $\Qcal$ is as defined in \eqref{eq:qcaldef},

 and the last inequality follows from the fact that $\epsa \leq \frac{1}{2}$. Thus, since by our choice of $\epsa$ we know that $\epsa \leq \pa{\frac{\epsr^2}{100\Qcal}}^4$, it follows that
\begin{equation*}
\abs{\zeta_k(\rho^-) - \hzeta_k(\rho^-)} \leq \frac{\epsr^2}{100}.
\end{equation*}

For the sake of clarity, we assume $\Rcal \geq 1$ -- otherwise, we can choose $M = O(\log(1 / \epsr))$, and a similar analysis holds.  Taken together with \eqref{eq:rhostardiff} and the fact that $\abs{\rho^- - \rho_k^*} \leq \frac{\epsr^{3}}{100\Rcal}$ and $\epsr \leq 1$, we have that
\begin{align*}
\rho^- \geq \rho_k^* - \frac{\epsr^{3}}{100\Rcal} = \zeta_k(\rho_k^*) - \frac{\epsr^{3}}{100\Rcal} \geq \zeta_k(\rho_k^*) - \frac{\epsr^2}{100} \geq \zeta_k(\rho^-) - \frac{11\epsr^2}{100} \geq \hzeta_k(\rho_k^-) - \frac{12\epsr^2}{100}.
\end{align*}
Note that, by a similar reasoning as above, it must be the case that $\hzeta_k(\rho^-) \geq \frac{\epsr}{2}$.
Since we have ensured throughout the procedure that $\rho^- \leq \hzeta_k(\rho^-)$, it follows that
\begin{equation*}
(1-\epsr)\hzeta_k(\rho^-) \leq \rho^- \leq (1+\epsr)\hzeta_k(\rho^-),
\end{equation*}
as desired, and so we set $\rho_k = \rho^-$.
\end{proof}

\subsection{Proof of Lemma \ref{lem:abinductapprox}.}
\begin{proof}
By sufficiently small, we mean that $\epsa > 0$ is chosen such that $\epsa \leq \min\braces{\pa{\frac{\epsr^2}{100\Qcal}}^4, \pa{\frac{\epsr^2}{100\Wcal}}^4, \frac{1}{2}}$, where $\epsr$ is as defined in the algorithm.

Following the standard line of reasoning, as presented by \cite{nesterov2018lectures}, we proceed via proof by induction. For $k=0$,
\begin{equation*}
A_0 f(x_0) + B_0 = \min_{x \in \reals^d} \psi_0 = 0,\quad f(x_0) \leq \Fcal,\quad \norm{v_0 - x^*}_\bB^2 = \norm{x_0 - x^*}_\bB^2,\quad \text{ and }\quad v_0 = x_0 \in \Lcal.
\end{equation*}

Now suppose, for some $k \geq 0$, that \eqref{eq:abinductapprox} and \eqref{eq:fvaldec} hold. 
To show that $\rhoin^+ = \mP$ is a valid upper bound on $\rho_k^*$, we note that for any $\tau \in [0,1]$, letting $y_k = (1-\tau)x_k + \tau v_k$, $f(y_k) \leq \max\braces{f(x_k), f(v_k)} \leq \max\braces{\Fcal, f(v_k)}$, by our inductive assumption. We also know by our inductive assumption that $\norm{v_k - x^*}_\bB^2 \leq \norm{x_0 - x^*}_\bB^2$. Thus, since 
\begin{equation*}
\norm{v_k - x_0}_\bB^2 \leq 2\norm{v_k-x^*}_\bB^2 + 2\norm{x_0 - x^*}_\bB^2 \leq 4\norm{x_0 - x^*}_\bB^2,
\end{equation*}
it follows that $v_k \in \K$, which means that $f(v_k) \leq \Fcal$, and so $f(y_k) \leq \Fcal$. It follows that, for all $\tau \in [0,1]$, $\norm{T_{\bB}(y_k) - y_k}_\bB^2 \leq \mP$, where $\mP$ is defined as in \eqref{eq:pdef}, since $f(T_{\bB}(y_k)) \leq f(y_k)$ for all $x \in \reals^d$. Thus, $\mP$ is a valid upper bound on $\rho_k^*$. 

For the lower bound on $\rho_k^*$, we note that, based on the condition for when the $\rhosearch$ procedure is reached in $\fastq$, it must be the case that $\rhoin^- \leq (1+\epsf)\norm{x_{k+1}^- - y_k^-}_\bB^2$ and $\rhoin^- \leq \norm{x_{k+1}^- - y_k^-}_\bB^2 - \Qcal\epsa^{1/4}$. Thus, from \eqref{eq:c2diffbound}, it can be seen that $\rhoin^- \leq \zeta(\rhoin^-)$, and so it follows that $\rhoin^- \leq \rho_k^*$. Therefore, the correctness of $\rhosearch$ can be ensured.

With this observation in hand, we may see that, for any $x \in \reals^d$,
\begin{align*}
\psi_{k+1}(x) &\geq \psi_k^* + \frac{1}{2}\norm{x-v_k}_\bB^2 + a_{k+1}\bra{f(x_{k+1}) + \inner{\grad f(x_{k+1}), x - x_{k+1}}}\\
&\geq A_k f(x_k) + B_k + \frac{1}{2}\norm{x-v_k}_\bB^2 + a_{k+1}\bra{f(x_{k+1}) + \inner{\grad f(x_{k+1}), x - x_{k+1}}}\\
&\geq A_k (f(x_{k+1}) + \inner{\grad f(x_{k+1}), x_k - x_{k+1}}) + B_k + \frac{1}{2}\norm{x-v_k}_\bB^2 + a_{k+1}\bra{f(x_{k+1}) + \inner{\grad f(x_{k+1}), x - x_{k+1}}}\\
&= A_{k+1}f(x_{k+1}) + B_k + \frac{1}{2}\norm{x-v_k}_\bB^2 + \inner{\grad f(x_{k+1}), A_k(x_k-x_{k+1}) + a_{k+1}(x - x_{k+1})}\\
&= A_{k+1}f(x_{k+1}) + B_k + \frac{1}{2}\norm{x-v_k}_\bB^2 + \inner{\grad f(x_{k+1}), a_{k+1}(x-v_{k}) + A_{k+1}(y_k - x_{k+1})},
\end{align*}
where the last equalities is due to the fact that $A_{k+1}y_k = A_k x_k + a_{k+1}v_k$. Note that
\[ \min\limits_{x \in \reals^d} \frac{1}{2}\norm{x-v_k}_\bB^2 + \inner{\grad f(x_{k+1}), a_{k+1}(x-v_{k})} = -\frac{a_{k+1}^2}{2}\norm{\grad f(x_{k+1})}_{\bB^{-1}}^2. \]
Combining this observation with Lemma \ref{lem:fsgradineq}, the fact that $\rhoin^- \leq \norm{x_{j+1}^- - y_j^-}_\bB^2$, and our choice of $\epsa$, we have
\begin{align*}
\min\limits_{x \in \reals^d} \psi_{k+1}(x) &\geq A_{k+1}f(x_{k+1}) + B_k - \frac{a_{k+1}^2}{2}\norm{\grad{f(x_{k+1})}}_{\bB^{-1}}^2 + \inner{\grad f(x_{k+1}), A_{k+1}(y_k - x_{k+1})}\\
&\geq A_{k+1}f(x_{k+1}) + B_k - \frac{A_{k+1}}{2L_3\rho_k}\norm{\grad{f(x_{k+1})}}_{\bB^{-1}}^2 \\
&\quad\quad+ A_{k+1}\pa{\frac{1}{2L_3 \hat{r}_{\bB}^{2}(x_{k+1},y_k)} \norm{\grad f(x_{k+1})}_{\bB^{-1}}^2 + \frac{3L_3}{8} \hat{r}_{\bB}^{4}(x_{k+1},y_k) - \frac{\Wcal\epsa^{1/4}}{\rhoin^-}}\\
&\geq A_{k+1}f(x_{k+1}) + B_k - \frac{A_{k+1}}{2L_3\rho_k}\norm{\grad{f(x_{k+1})}}_{\bB^{-1}}^2 \\
&\quad\quad+ A_{k+1}\pa{\frac{1}{2L_3 \hat{r}_{\bB}^{2}(x_{k+1},y_k)} \norm{\grad f(x_{k+1})}_{\bB^{-1}}^2 + \frac{3L_3}{8} \hat{r}_{\bB}^{4}(x_{k+1},y_k) - \frac{\epsr^2}{100\rhoin^-}}.
\end{align*}
 We also know, by the guarantees of $\rhosearch$ in Theorem \ref{thm:rhosearch}, along with the choice of $\epsa$, that $\rho_k \geq (1-\epsr)\hzeta(\rho_k) = (1-\epsr)\hat{r}_{\bB}^{2}(x_{k+1},y_k)$, and so
\begin{align*}
\min\limits_{x \in \reals^d} \psi_{k+1}(x) &\geq A_{k+1}f(x_{k+1}) + B_k - \frac{A_{k+1}}{2L_3(1-\epsr)\hat{r}_{\bB}^{2}(x_{k+1},y_k)}\norm{\grad{f(x_{k+1})}}_{\bB^{-1}}^2 \\
&\quad\quad+ A_{k+1}\pa{\frac{1}{2L_3 \hat{r}_{\bB}^{2}(x_{k+1},y_k)} \norm{\grad f(x_{k+1})}_{\bB^{-1}}^2 + \frac{3L_3}{8} \hat{r}_{\bB}^{4}(x_{k+1},y_k) - \frac{\epsr^2}{100\rhoin^-}}\\
&\geq A_{k+1}f(x_{k+1}) + B_k + A_{k+1}\pa{\frac{3L_3}{8} \hat{r}_{\bB}^{4}(x_{k+1},y_k) - \epsc},
\end{align*}
where
\begin{equation*}
\epsc \defeq \frac{\epsr}{2L_3(1-\epsr)\rhoin^-}\norm{\grad{f(x_{k+1})}}_{\bB^{-1}}^2 + \frac{\epsr^2}{100\rhoin^-}.
\end{equation*}
Therefore, by our choice of $\epsr \leq \frac{3L_3\rhoin^- \mP^2}{16\Tcal}$, where
\begin{equation}\label{eq:tcaldef}
\Tcal \defeq \frac{\mG}{L_3} + \frac{1}{100},
\end{equation}
 \eqref{eq:abinductapprox} holds for $k+1$, proving the induction step.
In addition, we may note that
\begin{align*}
\psi_{k+1}(x) &= \frac{1}{2}\norm{x-x_0}_\bB^2 + \sum\limits_{i=0}^{k+1} a_i\bra{f(x_{k+1}) + \inner{\grad f(x_{k+1}), x - x_{k+1}}} \leq \frac{1}{2}\norm{x-x_0}_\bB^2 + \sum\limits_{i=0}^{k+1} a_i f(x) \\
&= A_{k+1} f(x) + \frac{1}{2}\norm{x-x_0}_\bB^2.
\end{align*}
Since $v_{k+1} = \argmin\limits_{x \in \reals^d}\psi_{k+1}(x)$ and $\psi_{k+1}(x)$ is a quadratic function, it follows that, for all $x \in \reals^d$,
\begin{align*}
\psi_{k+1}(x) &= \psi_{k+1}(v_{k+1}) + \inner{\grad \psi_{k+1}(v_{k+1}), x - v_{k+1}} + \frac{1}{2}\norm{x-v_{k+1}}_\bB^2 \\
&= \psi_{k+1}(v_{k+1}) + \frac{1}{2}\norm{x-v_{k+1}}_\bB^2\\
&\leq A_{k+1}f(x) + \frac{1}{2}\norm{x-x_0}_\bB^2.
\end{align*}
Taken together, this gives us that
\begin{align*}
A_{k+1}f(x_{k+1}) + B_{k+1} + \frac{1}{2}\norm{x-v_{k+1}}_\bB^2 &\leq \min\limits_{x \in \reals^d} \psi_{k+1}(x) + \frac{1}{2}\norm{x-v_{k+1}}_\bB^2 \\
&= \psi_{k+1}(v_{k+1}) + \frac{1}{2}\norm{x-v_{k+1}}_\bB^2\\
&\leq A_{k+1}f(x) + \frac{1}{2}\norm{x-x_0}_\bB^2.
\end{align*}
Rearranging and letting $x = x^*$, we have that
\begin{equation*}
A_{k+1}(f(x_{k+1}) - f(x^*)) + B_{k+1} + \frac{1}{2}\norm{x^*-v_{k+1}}_\bB^2 \leq \frac{1}{2}\norm{x^*-x_0}_\bB^2,
\end{equation*}
and so it follows that
\begin{equation*}
\norm{v_{k+1}-x^*}_\bB^2 \leq \norm{x_0-x^*}_\bB^2
\end{equation*}
and $v_{k+1}, x_{k+1} \in \Lcal$.
\end{proof}

\subsection{Proof of Corollary \ref{cor:akupperbound}.}
\begin{proof}
Note that, for all $k \geq 0$, $x \in \reals^d$,
\begin{align*}
\psi_k(x) = \frac{1}{2}\norm{x-x_0}_\bB^2 + \sum\limits_{i=0}^k a_i\bra{f(x_k) + \inner{\grad f(x_k), x - x_k}} \leq \frac{1}{2}\norm{x-x_0}_\bB^2 + \sum\limits_{i=0}^k a_i f(x) = A_k f(x) + \frac{1}{2}\norm{x-x_0}_\bB^2,
\end{align*}
and so it follows from Lemma \ref{lem:abinductapprox} that
\begin{equation*}
A_k f(x_k) + B_k \leq \min\limits_{x \in \reals^d} \psi_k(x) \leq \min\limits_{x \in \reals^d} A_k f(x) + \frac{1}{2}\norm{x - x_0}_\bB^2 = A_k f(x^*) + \frac{1}{2}\norm{x^* - x_0}_\bB^2.
\end{equation*}
Rearranging, we have
\begin{align}
\frac{3L_3}{16}\sum\limits_{i=0}^{k-1} A_{i+1} \hat{r}_{\bB}^{4} (x_{i+1},y_i) &= B_k \leq A_k(f(x^*) - f(x_k)) + \frac{1}{2}\norm{x^* - x_0}_\bB^2 \leq \frac{1}{2}\norm{x^* - x_0}_\bB^2 \nonumber
\end{align}
and so
\begin{equation*}
f(x_k) - f(x^*) \leq \frac{1}{2A_k} \norm{x^* - x_0}_\bB^2.
\end{equation*}
\end{proof}

\subsection{Proof of Lemma \ref{lem:akboundrho}.}
\begin{proof}
Note that, by our choice of $A_k$ and $a_k$,
\begin{equation}\label{eq:rootdiffineq}
A_{k+1}^{1/2} - A_k^{1/2} = \frac{a_{k+1}}{A_{k+1}^{1/2} + A_k^{1/2}} = \frac{1}{A_{k+1}^{1/2} + A_k^{1/2}}\sqrt{\frac{A_{k+1}}{L_3 \rho_k}} \geq \sqrt{\frac{1}{4L_3\rho_k}}.
\end{equation}
Again, we procede with a proof by induction. $A_0 = 0$, thus the case for $k=0$ holds. Now, suppose for some $k \geq 0$,
\begin{equation*}
A_k \geq \frac{1}{4L_3}\pa{\sum\limits_{i=0}^{k-1} \frac{1}{\rho_i^{1/2}}}^2.
\end{equation*}
By \eqref{eq:rootdiffineq}, we know that
\begin{equation*}
A_{k+1}^{1/2} \geq A_k^{1/2} + \sqrt{\frac{1}{4L_3\rho_k}} \geq \sqrt{\frac{1}{4L_3}}\sum\limits_{i=0}^{k-1} \frac{1}{\rho_i^{1/2}} + \sqrt{\frac{1}{4L_3\rho_k}} = \sqrt{\frac{1}{4L_3}}\sum\limits_{i=0}^{k} \frac{1}{\rho_i^{1/2}}
\end{equation*}
which concludes the induction step.

\subsection{Proof of Lemma \ref{lem:aklowerboundp3}.}
Using Theorem \ref{thm:rhosearch} and the fact that $\epsr < 1$, we have that $\rho_i \leq 2\hat{r}_{\bB}^2 (x_{i+1},y_i)$. By Lemma \ref{lem:akboundrho}, it follows that, for all $k \geq 0$,
\begin{equation}\label{eq:aklowerbound}
A_k \geq \frac{1}{4L_3}\pa{\sum\limits_{i=0}^{k-1} \frac{1}{\rho_i^{1/2}}}^2 \geq \frac{1}{8L_3}\pa{\sum\limits_{i=0}^{k-1} \frac{1}{\hat{r}_{\bB} (x_{i+1},y_i)}}^2.
\end{equation}
Note that, for all $k \geq 0$, $x \in \reals^d$,
\begin{align*}
\psi_k(x) = \frac{1}{2}\norm{x-x_0}_\bB^2 + \sum\limits_{i=0}^k a_i\bra{f(x_k) + \inner{\grad f(x_k), x - x_k}} \leq \frac{1}{2}\norm{x-x_0}_\bB^2 + \sum\limits_{i=0}^k a_i f(x) = A_k f(x) + \frac{1}{2}\norm{x-x_0}_\bB^2,
\end{align*}
and so it follows that
\begin{equation*}
A_k f(x_k) + B_k \leq \min\limits_{x \in \reals^d} \psi_k(x) \leq \min\limits_{x \in \reals^d} A_k f(x) + \frac{1}{2}\norm{x - x_0}_\bB^2 = A_k f(x^*) + \frac{1}{2}\norm{x^* - x_0}_\bB^2.
\end{equation*}
Rearranging, we have
\begin{equation}\label{eq:lmconstraint}
\frac{3L_3}{16}\sum\limits_{i=0}^{k-1} A_{i+1} \hat{r}_{\bB}^{4} (x_{i+1},y_i) = B_k \leq A_k(f(x^*) - f(x_k)) + \frac{1}{2}\norm{x^* - x_0}_\bB^2 \leq \frac{1}{2}\norm{x^* - x_0}_\bB^2.
\end{equation}
The objective now is to lower bound the quantity $\sum\limits_{i=0}^{k-1} \frac{1}{\hat{r}_{\bB} (x_{i+1},y_i)}$ from \eqref{eq:aklowerbound}, subject to the constraint given by \eqref{eq:lmconstraint}. After defining $\xi_i \defeq \hat{r}_{\bB}(x_{i+1},y_i)$ and $D \defeq \frac{8}{3L_3}\norm{x_0 - x^*}_\bB^2$, our aim is to minimize
\begin{equation*}
\min\limits_{\xi \in \reals^k}\braces{\sum\limits_{i=0}^{k-1}\frac{1}{\xi_i} \ : \quad \sum\limits_{i=0}^{k-1}A_{i+1}\xi_i^{4} \leq D}.
\end{equation*}
We may introduce a Lagrange multiplier $\lambda$, giving us the following optimality conditions:
\begin{equation*}
\frac{1}{\xi_i^2} = \lambda A_{i+1}\xi_i^3, \quad \ i\in\braces{0,\dots,k-1}.
\end{equation*}
Therefore, $\xi_i = \pa{\frac{1}{\lambda A_{i+1}}}^{1/5}$. This gives us
\begin{equation*}
D = \sum\limits_{i=0}^{k-1} A_{i+1}\pa{\frac{1}{\lambda A_{i+1}}}^{4/5} = \frac{1}{\lambda^{4/5}}\sum\limits_{i=0}^{k-1} A_{i+1}^{1/5}.
\end{equation*} 
Thus, we have
\begin{equation*}
\xi^* = \sum\limits_{i=0}^{k-1} (\lambda A_{i+1})^{1/5} = \frac{1}{D^{1/4}}\pa{\sum\limits_{i=0}^{k-1} A_{i+1}^{1/5}}^{5/4},
\end{equation*}
and so
\begin{equation*}
\sum\limits_{i=0}^{k-1} \frac{1}{\hat{r}_{\bB}^{1/2}(y_i)} \geq \frac{1}{D^{1/4}}\pa{\sum\limits_{i=0}^{k-1} A_{i+1}^{1/5}}^{5/4}.
\end{equation*}
It follows that
\begin{equation*}
A_k \geq \frac{1}{8L_3D^{1/2}}\pa{\sum\limits_{i=1}^{k} A_i^{1/5}}^{5/2}, \quad k \geq 1.
\end{equation*}
Let $\theta = \frac{1}{8L_3D^{1/2}}$ and $C_k = \pa{\sum\limits_{i=1}^{k} A_i^{1/5}}^{1/2}$. Then, we have that
\begin{equation*}
C_{k+1}^{2} - C_{k}^2 \geq \theta^{1/5}C_{k+1}.
\end{equation*}
Thus, we have that $C_03253 \geq \theta^{1/5}$, $C_{k+1} \geq C_k$, and so 
\begin{align*}
\theta^{1/5}C_{k+1} &\leq (C_{k+1}-C_{k})(C_{k+1}+C_{k})\\
 &\leq 2C_{k+1}(C_{k+1}-C_k).
\end{align*}
Thus, it follows that $C_k \geq \theta^{1/5}(1 + \frac{1}{2}(k-1))$ for all $k \geq 1$. Taken together, this gives us that
\begin{equation*}
A_k \geq \theta\pa{C_k^{2}}^{5/2} \geq \theta\pa{\theta^{1/5}\frac{k+1}{2}}^{5} = \theta^2\pa{\frac{k+1}{2}}^5 = \frac{3}{256L_3\norm{x_0-x^*}_\bB^2}\pa{\frac{k+1}{2}}^5.
\end{equation*}
\end{proof}

\subsection{Proof of Theorem \ref{thm:outofrhobound}.}
\begin{proof}
By the algorithm statement, we have that $\epsf = \min{\braces{\frac{3L_3^2\mP \rhoin^-}{32 \mG}, \frac{1}{2}}}$. By $\epsa > 0$ sufficiently small, we mean that 
\begin{equation*}
\epsa \leq \min\braces{\pa{\frac{\epsf}{\Vcal(1+\epsf)}}^4}.
\end{equation*}
 For both cases (a) and (b), it holds by Lemma \ref{lem:abinductapprox} (and the statement of this lemma) that

\begin{equation*}
A_{i}f(x_{i}) + B_i \leq \psi_i^* \defeq \min_{x \in \reals^d} \psi_i(x).
\end{equation*}
We begin by considering the case where (a) holds. We first observe that, since $f(\cdot)$ is convex, we have that, for all $z \in \Lcal$,
\begin{equation*}
f(z) - f(x^*) \leq \mP^{1/2}\norm{\grad f(z)}_{\bB^{-1}}.
\end{equation*}
If $\norm{\grad f(x_{k+1})}_{\bB^{-1}}^2 < \frac{\eps^2}{\mP}$, then we are done, as $f(z) - f(x^*) \leq \eps$, so we consider the case where $\norm{\grad f(x_{k+1})}_{\bB^{-1}}^2 \geq \frac{\eps^2}{\mP}$.

Thus, by Lemma \ref{lem:epsagradineq}, we have that
\begin{equation*}
\inner{\grad f(x_{k+1}), y_k - x_{k+1}} \geq \frac{1 - \Vcal\epsa^{1/4}}{2L_3 \hat{r}_{\bB}^{2}(x_{k+1},y_k)} \norm{\grad f(x_{k+1})}_{\bB^{-1}}^2 + \frac{3L_3}{8} \hat{r}_{\bB}^{4}(x_{k+1},y_k),
\end{equation*}
where $\Vcal \defeq \max\limits_{x,y\in \Lcal}\frac{6Z(x,y)W(x,y)\mP}{\eps^2 L_3^{1/4}}$. 

Since $\rhoin^- > (1+\epsf)\norm{x_{i+1}^- - y_i^-}_\bB^2 = (1+\epsf)\hat{r}_\bB^2(x_{i+1},y_i)$ (by (a)), we may follow the same approach as before to arrive at
\begin{align*}
\min\limits_{x \in \reals^d} \psi_{i+1}(x) &\geq A_{i+1}f(x_{i+1}) + B_i - \frac{a_{i+1}^2}{2}\norm{\grad{f(x_{i+1})}}_{\bB^{-1}}^2 + \inner{\grad f(x_{i+1}), A_{i+1}(y_i - x_{i+1})}\\
&\geq A_{i+1}f(x_{i+1}) + B_i - \frac{A_{i+1}}{2L_3\rhoin^-}\norm{\grad{f(x_{i+1})}}_{\bB^{-1}}^2 \\
&\quad\quad+ A_{i+1}\pa{\frac{1 - \Vcal\epsa^{1/4}}{2L_3 \hat{r}_{\bB}^{2}(x_{i+1},y_i)} \norm{\grad f(x_{i+1})}_{\bB^{-1}}^2 + \frac{3L_3}{8} \hat{r}_{\bB}^{4}(x_{i+1},y_i)}\\
&> A_{i+1}f(x_{i+1}) + B_i - \frac{A_{i+1}}{2L_3(1+\epsf)\hat{r}_\bB^2(x_{i+1},y_i)}\norm{\grad{f(x_{i+1})}}_{\bB^{-1}}^2 \\
&\quad\quad+ A_{i+1}\pa{\frac{1 - \Vcal\epsa^{1/4}}{2L_3 \hat{r}_{\bB}^{2}(x_{i+1},y_i)} \norm{\grad f(x_{i+1})}_{\bB^{-1}}^2 + \frac{3L_3}{8} \hat{r}_{\bB}^{4}(x_{i+1},y_i)}\\
&= A_{i+1}f(x_{i+1}) + B_i + A_{i+1}\pa{\frac{\pa{(1+\epsf)\pa{1 - \Vcal\epsa^{1/4}} - 1}\norm{\grad f(x_{i+1})}_{\bB^{-1}}^2}{2L_3 \hat{r}_{\bB}^{2}(x_{i+1},y_i)} + \frac{3L_3}{8} \hat{r}_{\bB}^{4}(x_{i+1},y_i)}.
\end{align*}
Thus, since $\epsf = \min{\braces{\frac{3L_3^2\mP \rhoin^-}{32 \mG}, \frac{1}{2}}}$ and $\epsa \leq \pa{\frac{\epsf}{\Vcal(1+\epsf)}}^4$, it follows that 
\begin{align*}
\min\limits_{x \in \reals^d} \psi_{i+1}(x) &\geq A_{i+1}f(x_{i+1}) + B_i + \frac{3L_3A_{i+1}}{8} \hat{r}_{\bB}^{4}(x_{i+1},y_i) = A_{i+1}f(x_{i+1}) + B_{i+1}.
\end{align*}
As before, we may observe that
\begin{align*}
\psi_{i+1}(x) &= \frac{1}{2}\norm{x-x_0}_\bB^2 + \sum\limits_{j=0}^{i+1} a_j\bra{f(x_{i+1}) + \inner{\grad f(x_{i+1}), x - x_{i+1}}} \leq \frac{1}{2}\norm{x-x_0}_\bB^2 + \sum\limits_{j=0}^{i+1} a_j f(x) \\
&= A_{i+1} f(x) + \frac{1}{2}\norm{x-x_0}_\bB^2,
\end{align*}
and so it follows that
\begin{equation*}
f(x_{i+1}) - f(x^*) \leq \frac{1}{2A_{i+1}}\norm{x_0-x^*}_\bB^2.
\end{equation*}
By Lemma \ref{lem:akboundrho}, we know that $A_{i+1} \geq \frac{1}{4L_3\rhoin^-}$, and so it follows that 
\begin{equation*}
f(x_{i+1}) - f(x^*) \leq 2L_3\rhoin^-\norm{x_0-x^*}_\bB^2.
\end{equation*}

We now consider the case where (b) holds, i.e., $\rhoin^- \leq (1+\epsf)\norm{x_{k+1}^- - y_k^-}_\bB^2$ and $\rhoin^- > \norm{x_{k+1}^- - y_k^-}_\bB^2 - \Qcal\epsa^{1/4}$. We may observe that
\begin{equation*}
\norm{x_{k+1}^- - y_k^-}_\bB^2 \geq \frac{\rhoin^-}{1+\epsf},
\end{equation*}
and so, if we choose $\epsa \leq \pa{\frac{\epsf\rhoin^-}{\Qcal(1+\epsf)}}^4$, it follows that
\begin{equation*}
\rhoin^- > \norm{x_{k+1}^- - y_k^-}_\bB^2 - \Qcal\epsa^{1/4} \geq \norm{x_{k+1}^- - y_k^-}_\bB^2 - \frac{\epsf\rhoin^-}{(1+\epsf)} \geq (1-\epsf)\norm{x_{k+1}^- - y_k^-}_\bB^2,
\end{equation*}
and so we have that
\begin{equation*}
(1-\epsf)\norm{x_{k+1}^- - y_k^-}_\bB^2 \leq \rhoin^- \leq (1+\epsf)\norm{x_{k+1}^- - y_k^-}_\bB^2.
\end{equation*}

Following a line of reasoning as before, we may use Lemma \ref{lem:fsgradineq} with $c = (1+\epsf)$, along with the fact that $\rhoin^- \geq (1-\epsf)\norm{x_{i+1}^- - y_i^-}_\bB^2$, to see that
\begin{align*}
\min\limits_{x \in \reals^d} \psi_{i+1}(x) &\geq A_{i+1}f(x_{i+1}) + B_i - \frac{a_{i+1}^2}{2}\norm{\grad{f(x_{i+1})}}_{\bB^{-1}}^2 + \inner{\grad f(x_{i+1}), A_{i+1}(y_i - x_{i+1})}\\
&\geq A_{i+1}f(x_{i+1}) + B_i - \frac{A_{i+1}}{2L_3\rho_i}\norm{\grad{f(x_{i+1})}}_{\bB^{-1}}^2 \\
&\quad\quad+ A_{i+1}\pa{\frac{1}{2L_3 \hat{r}_{\bB}^{2}(x_{i+1},y_i)} \norm{\grad f(x_{i+1})}_{\bB^{-1}}^2 + \frac{3L_3}{8} \hat{r}_{\bB}^{4}(x_{i+1},y_i) - \frac{\Wcal\epsa^{1/4}}{(1+\epsf)\rhoin^-}}\\
&\geq A_{i+1}f(x_{i+1}) + B_i - \frac{A_{i+1}}{2L_3(1-\epsf)\hat{r}_{\bB}^{2}(x_{i+1},y_i)}\norm{\grad{f(x_{i+1})}}_{\bB^{-1}}^2 \\
&\quad\quad+ A_{i+1}\pa{\frac{1}{2L_3 \hat{r}_{\bB}^{2}(x_{i+1},y_i)} \norm{\grad f(x_{i+1})}_{\bB^{-1}}^2 + \frac{3L_3}{8} \hat{r}_{\bB}^{4}(x_{i+1},y_i) - \frac{\Wcal\epsa^{1/4}}{(1+\epsf)\rhoin^-}}\\
&= A_{i+1}f(x_{i+1}) + B_i + A_{i+1}\pa{\frac{3L_3}{8} \hat{r}_{\bB}^{4}(x_{i+1},y_i) - \epsch},
\end{align*}
where
\begin{equation*}
\epsch \defeq \frac{\epsf}{2L_3(1-\epsf)\rhoin^-}\norm{\grad{f(x_{k+1})}}_{\bB^{-1}}^2 + \frac{\Wcal\epsa^{1/4}}{(1+\epsf)\rhoin^-}.
\end{equation*}
Thus, for $\epsf = \min{\braces{\frac{3L_3^2\mP \rhoin^-}{32 \mG}, \frac{1}{2}}}$, and $\epsa \leq \pa{\frac{3L_3\mP^2\rhoin^-}{32\Wcal}}^4$, it follows that 
\begin{equation*}
\min\limits_{x \in \reals^d} \psi_{i+1}(x) \geq A_{i+1}f(x_{i+1}) + B_i + A_{i+1}\pa{\frac{3L_3}{16} \hat{r}_{\bB}^{4}(x_{i+1},y_i)} = A_{i+1}f(x_{i+1}) + B_{i+1}.
\end{equation*}
Therefore, it follows that
\begin{equation*}
f(x_{i+1}) - f(x^*) \leq \frac{1}{2A_{i+1}}\norm{x_0-x^*}_\bB^2,
\end{equation*}
and since by Lemma \ref{lem:akboundrho}, we know that $A_{i+1} \geq \frac{1}{4L_3\rhoin^-}$, we have that 
\begin{equation*}
f(x_{i+1}) - f(x^*) \leq 2L_3\rhoin^-\norm{x_0-x^*}_\bB^2.
\end{equation*}
\end{proof}

\subsection{Proof of Theorem \ref{thm:smoothrate}.}
\begin{proof}
Let $\Zcal \defeq \max\braces{\Acal, \mG, \mP, \Qcal, \Rcal, \Vcal, \Wcal, L_3}$. By appropriate initialization, we mean that $\rhoin^-$, $\epsa$ are chosen such that $\rhoin^-\leq\frac{\eps}{2L_3\mP}$, and 
\begin{align*}
\epsa &< \min\braces{\pa{\frac{\epsr^2}{100\Qcal}}^4, \pa{\frac{\epsr^2}{100\Wcal}}^4, \pa{\frac{\epsf}{\Vcal(1+\epsf)}}^4, \pa{\frac{\epsf\rhoin^-}{\Qcal(1+\epsf)}}^4, \pa{\frac{3L_3\mP^2\rhoin^-}{32\Wcal}}^4, \frac{1}{2}}\\
&\leq \min\braces{O\pa{poly\pa{\frac{\eps}{\Zcal}}}, \frac{1}{2}},
\end{align*}
where $\epsf$ and $\epsr$ are as defined in the $\fastq$ algorithm. Thus, based on our choices of $\rhoin^-$ and $\epsa$, the iteration complexity follows immediately from Theorems \ref{thm:smoothconvrate} and \ref{thm:outofrhobound}. Each iteration of $\fastq$ requires at most $O(\log(\frac{\Zcal}{\eps}))$ iterations of $\rhosearch$, each of which requires at most $O(\log(\frac{\Zcal}{\eps}))$ iterations of $\aam$, and each iteration of $\aam$ requires at most $O(\log^{O(1)}(\frac{\Zcal}{\eps}))$ calls to a gradient oracle and linear system solver. Taken together, this gives us a total computational cost of $O(\log^{O(1)}(\frac{\Zcal}{\eps}))$ calls to a gradient oracle and linear system solver per iteration of $\fastq$.
\end{proof}

\end{document}